\documentclass[AMA,LATO1COL]{WileyNJD-v2}

\articletype{Article Type}%

\received{26 April 2016}
\revised{6 June 2016}
\accepted{6 June 2016}

\usepackage{amsfonts, longtable}

\begin{document}

\title{
{Improvements in the  estimation of the Weibull tail coefficient--a comparative study}
}

\author[1,2]{L\'{i}gia Henriques-Rodrigues}

\author[3,4]{Frederico Caeiro*}

\author[5,6]{M. Ivette Gomes}

\authormark{\textsc{Henriques-Rodrigues et al.}}

\address[1]{\orgdiv{Department of Mathematics, School of Sciences and Technology}, \orgname{University of \'Evora}, \orgaddress{\state{\'Evora}, \country{Portugal}}}

\address[2]{\orgdiv{Research Center in Mathematics and Applications (CIMA)}, \orgname{University of \'Evora}, \orgaddress{\state{\'Evora}, \country{Portugal}}}

\address[3]{\orgdiv{NOVA School of Science and Technology (FCT NOVA)}, \orgname{NOVA University Lisbon}, \orgaddress{\state{Caparica}, \country{Portugal}}}

\address[4]{\orgdiv{Center for Mathematics and Applications
(CMA)}, \orgname{NOVA University Lisbon}, \orgaddress{\state{Caparica}, \country{Portugal}}}

\address[5]{\orgdiv{DEIO, Faculty of Sciences }, 
\orgname{University of  Lisbon (FCUL)}, \orgaddress{\state{Lisbon}, \country{Portugal}}}

\address[6]{\orgdiv{Center of Statistics and Applications}, \orgname{University of Lisbon (CEAUL)}, \orgaddress{\state{Lisbon}, \country{Portugal}}}

\corres{*Frederico Caeiro, Faculdade de Ci\^encias e Tecnologia, 2829-516 Caparica, Portugal.\ \email{fac@fct.unl.pt}}
%


\abstract[Abstract]{
{The Weibull tail-coefficient (WTC) plays a crucial role in extreme value statistics when dealing with Weibull-type tails.}
Several distributions, such as normal, Gamma, Weibull, and Logistic distributions, exhibit this type of tail {behaviour}.
The WTC, denoted by $\theta$, is a parameter in a right-tail function of the form
{$ \overline F(x) :=1-F(x) =: {\rm e}^{-H(x)}$, where $H(x)=-\ln(1-F(x))$ represents a regularly varying cumulative hazard function with an index of regular variation equal to 1/$\theta$,  $\theta\in\mathbb{R}^{+}$}. 
The commonly used {WTC-estimators} in literature 
are often defined as functions of the log-excesses, making them closely related to estimators of the extreme value index (EVI) for Pareto-type tails. For a positive EVI, the classical estimator is the {Hill estimator.} 
Generalized means  have been successfully employed in estimating the EVI, leading to {reduction of  bias and of} root mean square error for appropriate threshold values. In this study, we propose and investigate new classes of {WTC-estimators} based on power $p$ of the log-excesses 
within a second-order framework.
The performance of these new estimators is evaluated through a large-scale Monte-Carlo simulation study, comparing them with existing WTC-estimators available in the literature. 
}

\keywords{Log-excesses, Monte-Carlo simulation, Power mean-of-order $p$, Semi-parametric estimation, Statistics of Extremes, Weibull tail coefficient}



\maketitle

\section{Introduction and scope of the article}\label{sec1}

Let  $\underline{X}_n=(X_1,\ldots, X_n)$ denote a random sample of $n$ {{\it independent, identically distributed} (IID)}  or even stationary weakly dependent {{\it random variables} (RV’s)} with {{\it cumulative distribution function} (CDF)}  $F(x)=\mathbb{P}(X \leq x)$. We denote by $(X_{1:n}\leq \dots \leq X_{n:n})$ the sample of ascending {{\it order statistics} (OS’s) associated with} the available random sample $\underline{X}_n$. Let us {further 
assume}  that there exist sequences of real constants $\{a_{n}>0\}$ and $\{b_{n}\in\mathbb{R}\}$ such that, for the maximum $X_{n:n}:=\max(X_1, \dots, X_n)$, linearly normalized,
		\begin{equation}\nonumber
		\mathbb{P}\left(\frac{X_{n:n}-b_{n}}{a_{n}} \leq x\right) \mathop{\longrightarrow}^w_{{_{n\rightarrow\infty}}} {\rm EV}_{\xi}(x), 	
		\end{equation}
with $\rm{EV}_{\xi}(x)$ a non-degenerate limit distribution necessarily of the type of the {{\it extreme value}  {($EV$)} CDF} {(Fisher and Tippett\cite{fisher1928limiting}, Gnedenko \cite{gnedenko1943})}
		\begin{equation}\label{EV}
{EV}_{\xi}(x):=\left\{\begin{array}{lll}
			\exp\left(-(1+\xi x)^{-1/\xi}\right),\quad 1+\xi x >0, & \mbox{if} & \xi\not=0\\
			\exp(-\exp(-x)), \quad x\in\mathbb{R}, & \mbox{if} & \xi =0.
		\end{array}
		\right.
	\end{equation}
\vskip 0.3cm 
The {CDF} $F$ is said to belong to the domain of attraction for maxima of  {$EV_\xi$}, with $\xi\in\mathbb{R}$, and we write {$F\in{\cal D_M}(EV_\xi)$}. The parameter $\xi$  is the {{\it extreme-value index}} (EVI), one of the primary parameters of extreme events.  The EVI{,  $\xi$}, measures essentially the weight of the  right-tail $\overline F:=1-F$.
{If $\xi < 0$, the right tail is light, of a max-Weibull-type,
\begin{equation}
\Psi_\alpha(x)=\exp(-(-x)^\alpha),\quad  x\leq 0 \qquad (\alpha>0,~ \xi=-1/\alpha)
\label{Psi-alpha}
\end{equation} 
and $F$ has a finite right endpoint $(x^F := \sup\{x : F(x) < 1\} < +\infty)$. If $\xi> 0$, the right-tail is heavy, of a negative  polynomial type or of a Fréchet-type, 
\begin{equation*}
\Phi_\alpha(x)=\exp(-x^{-\alpha}), \quad x\geq 0 \qquad (\alpha>0, ~ \xi=1/\alpha), 
\label{Phi-alpha}
\end{equation*} 
and F has an infinite right endpoint. If $\xi= 0$, the right tail is of an exponential or Gumbel-type,
\begin{equation*}
\Lambda(x)=\exp(-\exp(-x)),\quad x\in \mathbb{R}  \qquad (\xi=0),
\label{Lambda}
\end{equation*} 
and the right endpoint can be either finite or infinite.} 
 The last case will be of interest in this paper and we shall work with light tailed distributions in the Gumbel domain of attraction exhibiting a Weibull type tail, i.e., 
 \begin{equation}\label{Weibull_tail}
\lim_{t\to \infty} \frac{\ln (1-F(c\,t))}{\ln (1-F(t))}=c^{1/\theta},
 \end{equation}
for all $c>0$ and where $\theta>0$ is the {{\it Weibull tail coefficient}} (WTC). Denoting by {${\cal RV}_{\alpha}$ the class of regularly varying functions at infinity with index $\alpha$} {(see Bingham {\it et al.} \cite{bingham1989regular})}, i.e., positive measurable function $g(.)$ such that $\lim_{t\to \infty} g(tx)/g(t)=x^{\alpha}$, for all $x >0$, then with {$H(t):=-\ln(1-F(t))$ the} cumulative hazard function, $H(t) \in \mathcal{RV}_{1/\theta}$, for $\theta>0$ and we may also write that {$\overline{F}(t)=e^{-H(t)}$}. This type of models are very important to model large claims in non-life insurance cases \cite{beirlant1992}, hydrology, meteorology \cite{dewet:hal-01312912} and include, among others the Weibull, the Gamma, the Logistic and the Normal distributions. 
{Indeed, the WTC is the parameter $\theta$ used in a Weibull CDF  (one of the possible limiting distributions of linearly normalized minima), of the type,
\begin{equation}
\Psi_\theta^\ast(x)=1-\Psi_\theta(-x), \quad x\geq 0,
\label{Psi-alpha-ast}
\end{equation}
with $\Psi_\alpha(x)$ given in \eqref{Psi-alpha}. The Weibull CDF  $\Psi^\ast_\theta((x-\lambda)/\delta), x\geq \lambda$, with $\Psi_\theta^\ast(x)$ given in \eqref{Psi-alpha-ast}, and $(\lambda\in\mathbb{R}, \delta>0)$, a vector of location and scale parameters,    is a CDF commonly employed in reliability engineering and survival analysis. The Weibull CDF is flexible and can model a wide range of failure times. In applications, such as reliability engineering, the WTC estimation is important for modeling the failure characteristics of a system accurately, understanding whether the system has a decreasing $(\theta<1)$, constant ($\theta=1$), or increasing ($\theta>1$) failure rate over time, being this crucial for making informed decisions about maintenance, replacement, and other reliability-related issues. { With an accurate estimation of the WTC, the estimation of other relevant parameters can be performed, such as the estimation of high quantiles and the estimation of probabilities of exceeding levels that surpass the observed values, which are also important in the decision-making process.}  Rather than placing ourselves under a parametric framework, and beyond it, a semi-parametric framework can be considered, and a more reliable estimation of the WTC is then possible, unless the fitting of the Weibull CDF is `almost perfect’. We thus suggest the use of such a semi-parametric framework} for the estimation of the WTC, where we will consider the $k$ upper observations above the random threshold $X_{n-k:n}$ and work with intermediate sequences of OS’s{, i.e.\ a value of $k$} such that 
\begin{equation}\label{intermediate}
k=k_n \to \infty, \quad k\in [1,n) \quad \mbox {and} \quad k=o(n), \quad \mbox{as} \quad n\to \infty.
\end{equation}
Let $F^\leftarrow(t) :=\inf\left\{x:F(x)\geq t\right\}$ denote the generalized inverse function of $F$ and $U(t)$ the {\it reciprocal right-tail quantile function} (RTQF), defined as 
\begin{equation}
U(t):=F^\leftarrow(1-1/t),\quad t\geq 1.    
\label{eq_quantile_function_U}
\end{equation}
For Weibull type-tails the following relations hold true for the RTQF of $H$, 
\begin{equation}\label{FOC}
U(e^t)=H^\leftarrow (t) \in RV_\theta \iff U(t)=(\ln  t)^\theta L(\ln t),
\end{equation}
with $L$ a slowly varying function, i.e., $L \in \mathcal{RV}_0$. The conditions \eqref{intermediate} and the {{\it first order condition}} (FOC), in \eqref{FOC}, are used in the semi-parametric estimation of the WTC to prove the consistency of the estimators. To derive the non-{degenerate} behaviour of the WTC-estimators we have to assume the validity of a {\it second-order condition} (SOC) that rules the rate of convergence of the slowly varying function $L(t)$, in \eqref{FOC},
\begin{equation}\label{SOC}
\lim_{t\to \infty} \frac{\ln L(tx)-\ln L(t)}{B(t)}=\frac{x^\beta -1}{\beta},
\end{equation}
where $\beta \leq 0$ and $|B| \in \mathcal{RV}_{\beta}$. The rate of convergence increases when $\beta$ decreases. Among the several models that satisfy the SOC, in \eqref{SOC}, we refer to Table 1 in {Gardes and Girard\cite{gardes2008estimation} and Dierckx {\it et al.} \cite{dierckx2009anewestimation}}, where values of $\beta$ and functions $B(t)$ are presented for some models with Weibull type tail. Apart from the Weibull distribution, where $\beta=-\infty$ {($B(t)=0$)}, all the other models presented have $\beta=-1$. Observing the behaviour of the function $B(t)$ that controls the bias, we can consider three types of models: (1) models where $B(t)=0$, as the Weibull{, which generalizes the standard Exponential model}; (2) models with $B(t) \propto t^\beta=t^{-1}$, as the Logistic and Gumbel models; (3) models with $B(t)\propto t^\beta \ln t=\ln t /t$, as the Gaussian and Gamma models. From Remark 1, in  {Caeiro {\it et al.}}\cite{caeiroetal-wtcmop}, and under EVI-estimation, the EVI, $\xi$, and the second-order shape parameter, $\rho$, are both null corresponding to Gumbel's type {right-tails} ($\Lambda(x)=\exp(-\exp(-x))), x \in \mathbb{R}$) that {entail} a penultimate behaviour (see {Gomes}\cite{gomes1978some,gomes2004penultimate} and  {Gomes and de Haan}\cite{gomes1999approximation} for further details) of the {Fréchet-type} if $\theta>1$ or the {max-Weibull-type} if $\theta<1$. 

\subsection{Semi-parametric EVI- and WTC-estimators}

Several approaches have been devised for the {semi-parametric} estimation of the WTC. Among them, we mention the pioneer work of {Berred}\cite{berred1991record}, where record values were used, and the works of {Broniatowsk\cite{broniatowski1993ontheestimation}, Beirlant {\it et al.}\cite{beirlant1995themean} and Beirlant and Teugels\cite{beirlant1996pratical}}, where the WTC-estimation was performed based on adequate functions of the {OS’s}. In this paper we are going to work with  two classes of WTC-estimators based on the log-excesses, 
\begin{equation}\label{Vik}
V_{ik}:= \ln \frac{X_{n-i+1:n}}{X_{n-k:n}}=\ln X_{n-i+1:n} - \ln X_{n-k:n}, \quad 1 \leq i \leq k <n, 
\end{equation}
and on the relative excesses, $U_{ik}$,  
\begin{equation}\label{Uik}
U_{ik}:= \frac{X_{n-i+1:n}}{X_{n-k:n}}, \quad 1 \leq i \leq k <n. 
\end{equation}
The log-excesses in \eqref{Vik} and the relative excesses in \eqref{Uik} are core functions of several {EVI-estimators. One of such classic estimators is the Hill estimator}  (H) \cite{hill1975}, which is given by the average of the {log-excesses, or equivalently, by the logarithm of the geometric mean of the relative excesses},
\begin{equation}\label{hill}
{\rm H}(k):=\frac{1}{k}\sum_{i=1}^k V_{ik}=\ln \left(\prod_{i=1}^k U_{ik} \right)^{1/k}, \quad 1 \leq k <n. 
\end{equation}
The asymptotic behaviour of the Hill estimator, in \eqref{hill}, and the trade-off between bias and variance exhibited by this classic EVI-estimator {has led}  researchers to improve the Hill estimator by considering new EVI-estimators based on the concept of {{\it generalized means} ({GM’s})}.

The literature {related to} GM EVI-estimators is diverse and several types of {GM’s} have been considered. In our study we will focus on two particular classes of {GM’s} based on the log-excesses, in \eqref{Vik}, and on the relative excesses, in \eqref{Uik}.  The first GM considered is the power mean of exponent $p$ of the log-excesses, PM$_p$, introduced in  {Gomes and Martins\cite{gomesmartins1999}}, also considered in {Segers\cite{segers2001residual}} and further studied in  {Gomes and Martins\cite{gomesmartins2001}}. This class of EVI-estimators is defined by 
\begin{equation}
	{\rm PM}_{k,n,p}\equiv 
	{\rm PM}_p(k):= 
	\left(	\frac{M_{k,n}^{(p)}}{\Gamma(p+1)}\right)^{1/p} \,\,, \quad {\rm for} \,\, {p>-1,\quad p\neq 0}
	\label{PM-EVI}
\end{equation}
with ${\rm PM}_1(k) \equiv {\rm H}(k)$, in \eqref{hill}, $\Gamma$ denoting the complete Gamma function and 
$$ M_{n,k}^{(p)}:= \frac{1}{k}\sum_{i=1}^k\left( \ln  {X_{n-i+1:n}}-\ln {X_{n-k:n}}\right)^p = \frac{1}{k}\sum_{i=1}^k V_{ik}^p,$$
with $V_{ik}$ defined in \eqref{Vik}. For further details on the asymptotic behaviour and the finite sample properties of the PM$_p$ EVI-estimators {see {Gomes and Martins} \cite{gomesmartins2001}.}

The second GM to be considered, the power mean-of-order $p$, denoted as H$_p$, was almost simultaneously introduced in 
Brilhante {\it et al.}\cite{brilhante2013asimple}, Paulauskas and Vai{\v{c}iulis}\cite{paulauskas2013improvement} 
and Beran {\it et al.}\cite{beran2013harmonic} (see also Segers\cite{segers2001residual} and Caeiro {\it et al.}\cite{cgbw2016}).
{The transformed-Hill (t-hill) estimator\cite{Stehlik2009,jordanova2012weak} is a special case of the GM H$_p$, where $p=-1$.}  
The functional form of this class of EVI-estimators is 
\begin{equation}
{\rm H}_{p}(k) :=\left\{
\begin{array}{cll}
\frac{1- \left(\frac{1}{k}\sum\limits_{i=1}^k  U_{ik}^p\right)^{-1}}{p},\quad  & \mbox{if} & p <1/\xi,\ p\neq 0,\\
{\rm H}(k), & \mbox{if} & p=0,
\end{array}
\right.
\label{Hill-p}
\end{equation} 
with $U_{ik}$ defined in \eqref{Uik}.
{For heavy-tailed models}, consistency and asymptotic normality are achieved under adequate conditions for values of $p<1/\xi$ and $p<1/(2\xi)$, respectively. The non-normal asymptotic behaviour of H$_p$ for $1/(2\xi)<p\leq 1/\xi$ was recently derived in 
{Gomes {\it et al.}}\cite{gomesetal-nonregular}.

The dependence on a tuning parameter $p$ makes the classes of EVI-estimators in \eqref{PM-EVI} and \eqref{Hill-p} very flexible allowing for a reduction in both bias and {\it root mean square error} (RMSE) for the EVI-estimation, for values of $p$ properly chosen. We will use these classes of GM EVI-estimators in the semi-parametric estimation of the WTC, improving the finite sample behaviour of the classic WTC-estimators.\\

The first class of classic semi-parametric WTC-estimators considered in this work was proposed in {Beirlant {\it et al.}}\cite{beirlant1996pratical} and further studied in {Girard\cite{girard2004} and Gardes and Girard\cite{gardes2006comparison}}, 
\begin{equation}\label{WTC-G}
    	\widehat \theta^{\rm G}(k):= \frac{ 
		{\rm H}(k)
	}{\frac{1}{k}\sum_{i=1}^k 
		\ln{\ln((n+1)/i)}-\ln{\ln((n+1)/(k+1))}}, \quad 1\leq k \leq n-1,
\end{equation}
with H$(k)$ the Hill estimator in \eqref{hill}.  In {Gardes and Girard}\cite{gardes2006comparison}, the authors studied the class of WTC-estimators of the type, 
\begin{equation}\label{WTC-est}
\widehat \theta (k):= \frac{{\rm H}(k)}{T_n} ,
\end{equation}
a Hill type WTC-estimators with three different choices for the positive sequence, $T_n$. The $\widehat \theta^{\rm G}(k)$ estimator, in \eqref{WTC-G}, is a member of the Hill type estimators in \eqref{WTC-est}  considering $T_n=\frac{1}{k}\sum_{i=1}^k 	\ln{\ln((n+1)/i)}-\ln{\ln((n+1)/(k+1))}$.
    
Generalizations and further improvements of the $\widehat \theta^{\rm G}(k)$ estimator can be found in {Gardes and Girard\cite{gardes2008estimation} and Goegebeur {\it et al.}\cite{goegebeur2010generalized}}. 
The second class considered in this paper was also studied in  {Gardes and Girard}\cite{gardes2006comparison} and has a functional form as in \eqref{WTC-est} {with $T_n=1/\ln(n/k)$,}
\begin{equation}\label{WTC-GG}
   \widehat \theta^{\rm GG} (k):= \ln(n/k) {\rm H}(k)
   , \quad 1\leq k \leq n-1,
\end{equation}
with H$(k)$ the Hill estimator in \eqref{hill}.\\

A first approach to the use of GM’s in the WTC-estimation was recently performed in {Caeiro {\it et al.}}\cite{caeiroetal-wtcmop, wtc-spe}, where the class of GM EVI-estimators H$_p$, in \eqref{Hill-p}, was plugged into the class of WTC-estimators in \eqref{WTC-GG}, 
\begin{equation}\label{WTC-Mop-GG}
       \widehat \theta_p^{\rm GG} (k):= \ln(n/k) {\rm H}_p(k), \quad 1\leq k \leq n-1, \quad {p\in \mathbb{R}}
\end{equation}
with $\widehat \theta_0^{\rm GG} (k) \equiv \widehat \theta^{GG} (k)$ in \eqref{WTC-GG}. The asymptotic properties of the new WTC-estimators were derived and the finite sample behaviour was assessed through a {Monte-Carlo} simulation study. For this class it was always possible to find a negative value of $p$ that enabled a reduction in bias and {RMSE}, for appropriate values of the threshold, $k$. 

In this work we are going to extend the previous results in {Caeiro {\it et al.}}\cite{caeiroetal-wtcmop, wtc-spe}, studying the asymptotic and finite sample behaviour of the following WTC-estimators:
\begin{eqnarray}
     	\widehat \theta^{\rm G}_p(k)&:=& 
      \frac{ 
		{\rm H}_p(k)}{\frac{1}{k}\sum_{i=1}^k 
		\ln{\ln((n+1)/i)}-\ln{\ln((n+1)/(k+1))}}, \quad 1\leq k \leq n-1, \quad {p\in \mathbb{R}},
  \label{WTC-Mop-G}\\
     \widetilde \theta_p^{GG} (k)&:= &\ln(n/k) {\rm PM}_p(k), \quad 1\leq k \leq n-1, \quad {p>-1,\quad p\neq 0,} 
     \label{WTC-PM-GG}
\end{eqnarray}
and 
\begin{equation}\label{WTC-PM-G}
 	\widetilde \theta^{\rm G}_p(k):= \frac{ 
		{\rm PM}_p(k)
	}{\frac{1}{k}\sum_{i=1}^k 
		\ln{\ln((n+1)/i)}-\ln{\ln((n+1)/(k+1))}}, \quad 1\leq k \leq n-1,\quad {p>-1,\quad p\neq 0,}
\end{equation}
with 
$\widehat \theta^{\rm G}_0(k)\equiv \widehat \theta^{\rm G}(k)$, in \eqref{WTC-G},\ 
$\widetilde \theta_1^{\rm GG} (k) \equiv \widehat \theta^{\rm GG} (k)$, in \eqref{WTC-GG},\ and\ 
$\widetilde \theta_1^{\rm G} (k) \equiv \widehat \theta^{\rm G} (k)$, in \eqref{WTC-G}. \\

In Section 2, we present {the asymptotic non-degenerate behaviour of the WTC-estimators under study assuming the validity of an adequate SOC and working with intermediate sequences, as in \eqref{intermediate}. {Particular classes of models, with  $B(t)=\alpha t^{-1}$, $\alpha \in \mathbb{R}$, and $B(t)=0$, are also considered and their asymptotic behaviour is derived}. For these classes of models some considerations are drawn {related to} the optimal values of $p$, i.e., the values of $p$ that cancel the asymptotic bias leading to reduced bias WTC-estimators. The asymptotic {MSE} of the estimators under study are also presented and some comparisons are made. Next, in Section 3, a {Monte-Carlo} simulation study is carried out to assess the performance of the WTC-estimators under study. The mean values, {RMSE’s}, simulated optimum levels, i.e., the $k$-levels that minimize the {RMSE’s} and the correspondent simulated optimum mean values and optimum {RMSE’s} are computed. Finally, in Section 4, we present some overall conclusions of this study.}

\section{Asymptotic non-degenerate behaviour}

In this section, and in order to derive the asymptotic normality of the WTC-estimators under study we assume the validity of the {SOC} in \eqref{SOC} and work with intermediate sequences of positive integers, $k$, as in \eqref{intermediate}.

\subsection{Asymptotic behaviour of the WTC-estimators under study}

In this subsection we present the limit distributions of the {WTC-estimators} under study, defined in \eqref{WTC-Mop-GG}, \eqref{WTC-Mop-G}, \eqref{WTC-PM-GG} and \eqref{WTC-PM-G}. The next theorem presents the asymptotic behaviour of the WTC-estimator $\widehat\theta^{GG}_p(k)$ derived in {Caeiro {\it et al.}}\cite{caeiroetal-wtcmop}. 

\begin{theorem}[Caeiro {\it et al.}\cite{caeiroetal-wtcmop}]\label{Theo-wtcmopGG}
    Under a second-order framework, i.e., assuming that condition 
    \eqref{SOC} holds, with $k$ an intermediate sequence, as in \eqref{intermediate}, the WTC-estimator $\widehat\theta^{\rm GG}_p(k)$, in \eqref{WTC-Mop-GG}, has an asymptotic distributional representation of the type,
	\begin{equation}\label{WTC-MOP-GG_dist}
		\widehat\theta^{\rm GG}_p(k)\stackrel{d}{=} \theta \left\{1+\frac{P_k}{\sqrt k}+\left(\frac{B\left(\ln (n/k)\right)}{\theta}-\frac{1}{\ln(n/k)} 
		\right)(1+{o_{_{\mathbb{P}}}}(1))\right\}, \quad{p \in {\mathbb R},}
	\end{equation}
 with $P_k$ a standard normal RV, frequently denoted by $\mathcal{N}(0,1)$. 
\end{theorem}

As noted in {Caeiro {\it et al.} }\cite{caeiroetal-wtcmop} the distributional representation of $\widehat\theta^{\rm GG}_p(k)$ in Theorem \ref{Theo-wtcmopGG} does not depend on the tuning parameter $p$ but a dependence on $p$ can appear if higher order terms are considered in the tail quantile function expansion. For details on the proof of Theorem \ref{Theo-wtcmopGG} see {Caeiro {\it et al.}}\cite{caeiroetal-wtcmop}.

The asymptotic distribution representation of the estimator $\widehat \theta^{\rm G}(k)$, in \eqref{WTC-G} can be {obtained} straightforwardly applying a result presented in {Girard}\cite{girard2004}. So, for intermediate $k$-values, {Girard}\cite{girard2004} proved that 
	$$\frac{1}{k}\sum_{i=1}^{k} \ln \ln (n/i) - \ln \ln (n/k) \sim \frac{1}{\ln (n/k)} \left( 1-\frac{1}{\ln (n/k)}\right).$$
	Then, 
	\begin{equation}\label{G-GG}
		\hat\theta^{\rm G}(k) \stackrel{d}{=} \hat\theta^{\rm GG}(k)\left(1+\frac{1}{\ln(n/k)}\right)(1+{o_{_{\mathbb{P}}}}(1)).
	\end{equation}

\begin{theorem}\label{Theo-wtcmopG}
    Under a second-order framework, i.e., assuming that condition 
    \eqref{SOC} holds, with $k$ an intermediate sequence, as in \eqref{intermediate}, the WTC-estimator $\widehat\theta^{\rm G}_p(k)$, in \eqref{WTC-Mop-G}, has an asymptotic distributional representation of the type,
	\begin{equation}\label{WTC-MOP-G_dist}
		\widehat\theta^{\rm G}_p(k)\stackrel{d}{=} \theta \left\{1+\frac{P_k}{\sqrt k}+\frac{B\left(\ln (n/k)\right)}{\theta} (1+{o_{_{\mathbb{P}}}}(1))\right\},
  \quad{p \in {\mathbb R},}
	\end{equation}
 with $P_k$ a standard normal {RV}. 
\end{theorem}

\begin{proof}
   The result in \eqref{G-GG} enable us to write 
   $$\hat\theta^{\rm G}_p(k) \stackrel{d}{=} \hat\theta^{\rm GG}_p(k)\left(1+\frac{1}{\ln(n/k)}\right)(1+{o_{_{\mathbb{P}}}}(1)).$$
   Given the asymptotic distributional representation of $\hat\theta^{\rm GG}_p(k)$, in \eqref{WTC-MOP-GG_dist}, the {result} follows. 
\end{proof}

\begin{corollary}\label{cor_wtcmop}
    	Under the conditions of Theorems \ref{Theo-wtcmopGG} and \ref{Theo-wtcmopG},  further assuming that  $\sqrt k B(\ln(n/k)) \to \lambda \in \mathbb{R}$, that additionally  $\sqrt k/ \ln(n/k) \to 0$ for the WTC-estimator $\widehat\theta^{\rm GG}_p(k)$,
	 and with $\bullet=\{GG,G\}$ denoting the WTC-estimators in \eqref{WTC-Mop-GG} and \eqref{WTC-Mop-G},
		\begin{equation*}
			\sqrt k \left(\widehat\theta^{\bullet}_p(k) -\theta\right) \xrightarrow[]{d} \mathcal{N}\left(\lambda, \theta^2\right){= \lambda+\theta \mathcal{N}(0,1)},\quad{p \in {\mathbb R}}.
		\end{equation*}
\end{corollary}
\begin{proof}
  The proof for the WTC-estimator $\widehat\theta^{\rm GG}_p(k)$ in \eqref{WTC-Mop-GG} can be seen in {Caeiro {\it et al.}}\cite{caeiroetal-wtcmop}.  From the distributional representation in \eqref{WTC-MOP-G_dist}, 
\begin{equation*}
\sqrt k \left(\widehat\theta^{\rm G}_p(k) -\theta \right)\stackrel{d}{=} \theta P_k+\sqrt k B\left(\ln (n/k)
\right)(1+{o_{_{\mathbb{P}}}}(1)){.}
\end{equation*}
Assuming that $\sqrt k B(\ln(n/k)) \to \lambda \in \mathbb{R}$, with $P_k$ the standard normal RV in \eqref{WTC-MOP-GG_dist}, denoted by $\mathcal{N}(0,1)$, the result follows.
\end{proof}

{Next, we present without proofs the asymptotic behaviour of 
$\widehat\theta^{GG}_p(k)$ and $\widehat\theta^{G}_p(k)$, in \eqref{WTC-Mop-GG} and \eqref{WTC-Mop-G}, respectively, for models with a function $B(t)$ given by $B(t)=\alpha\, t^{-1}$, with $\alpha \in \mathbb{R}${, and for models where $B(t)=0$}. }

\begin{corollary}\label{cor-wtcmopGG-new}
Under the conditions of Theorem \ref{Theo-wtcmopGG} and further assuming that
 $B(t)=\alpha\, t^{-1}$, with $\alpha \in \mathbb{R}$, the WTC-estimator $\widehat\theta^{\rm GG}_p(k)$, {$p \in {\mathbb R}$}, has an asymptotic distributional representation of the type,
	\begin{equation*}\label{WTC-MOP-GG_dist1}
		\widehat\theta^{\rm GG}_p(k)\stackrel{d}{=} \theta \left\{1+\frac{P_k}{\sqrt k}+\frac{\alpha -\theta}{\theta\ln(n/k)}(1+{o_{_{\mathbb{P}}}}(1))\right\},
	\end{equation*}
 with $P_k$ a standard normal RV. If the intermediate sequence $k$ is such that 
 $\sqrt k /\ln(n/k) \to \lambda \in \mathbb{R}$, 
		then,
		\begin{equation*}
			\sqrt k \left(\widehat\theta^{\rm GG}_p(k) -\theta\right) \xrightarrow[]{d} \mathcal{N}\left({\lambda \frac{\alpha-\theta}{\theta}}, \theta^2\right).
		\end{equation*}
 {If we assume that $B(t)=0$, the WTC-estimator $\widehat\theta^{\rm GG}_p(k)$ has an asymptotic distributional representation of the type,
	\begin{equation*}\label{WTC-MOP-GG_dist0}
		\widehat\theta^{\rm GG}_p(k)\stackrel{d}{=} \theta \left\{1+\frac{P_k}{\sqrt k}-\frac{1}{\ln(n/k)}(1+{o_{_{\mathbb{P}}}}(1))\right\},
	\end{equation*}
 and if $\sqrt k/\ln(n/k) \to \eta \in \mathbb{R}$ then, 
 \begin{equation*}
			\sqrt k \left(\widehat\theta^{\rm GG}_p(k) -\theta\right) \xrightarrow[]{d} \mathcal{N}\left(-\eta\; \theta, \theta^2\right).
		\end{equation*}}
\end{corollary}


\begin{corollary}\label{cor-wtcmopG-new}
Under the conditions of Theorem \ref{Theo-wtcmopG} and further  {assuming that}
$B(t)=\alpha\, t^{-1}$, with $\alpha \in \mathbb{R}$, the WTC-estimator $\widehat\theta^{\rm G}_p(k)$ with {$p \in {\mathbb R}$}, in \eqref{WTC-Mop-G}, has an asymptotic distributional representation of the type,
	\begin{equation*}\label{WTC-MOP-G_dist1}
		\widehat\theta^{\rm G}_p(k)\stackrel{d}{=} \theta \left\{1+\frac{P_k}{\sqrt k}+\frac{\alpha}{\theta\ln(n/k)}(1+{o_{_{\mathbb{P}}}}(1))\right\},
	\end{equation*}
 with $P_k$ a standard normal {RV.}
 If the intermediate sequence $k$ is such that {$\sqrt k B(\ln(n/k)) \to \lambda \in \mathbb{R}$}, 
		then,
		\begin{equation*}
			\sqrt k \left(\widehat\theta^{\rm G}_p(k) -\theta\right) \xrightarrow[]{d} \mathcal{N}\left(\lambda, \theta^2\right).
		\end{equation*}
   {If we assume that $B(t)=0$, the WTC-estimator $\widehat\theta^{\rm G}_p(k)$ has an asymptotic distributional representation of the type,
	\begin{equation*}\label{WTC-MOP-G_dist0}
		\widehat\theta^{\rm G}_p(k)\stackrel{d}{=} \theta \left\{1+\frac{P_k}{\sqrt k}+o_{_{\mathbb{P}}}(1)\right\},
	\end{equation*}
 and
 \begin{equation*}
			\sqrt k \left(\widehat\theta^{\rm GG}_p(k) -\theta\right) \xrightarrow[]{d} \mathcal{N}\left(0, \theta^2\right).
		\end{equation*}}
\end{corollary}

\vspace{0.2cm}
The next Theorem presents the asymptotic behaviour of the WTC-estimator $\widetilde \theta^{\rm GG}_p(k)$, in \eqref{WTC-PM-GG}. 

\begin{theorem}\label{Theo-wtcpmGG}
	Under a second-order framework, i.e., assuming that condition  \eqref{SOC} holds, with $k$ an intermediate sequence, as in \eqref{intermediate}, the WTC-estimator $\widetilde\theta^{\rm GG}_p(k)$ has an asymptotic distributional representation of the type,
\begin{equation}\label{WTC-MP-GG_dist}
		\widetilde\theta^{\rm GG}_p(k)\stackrel{d}{=} \theta \left\{1+{\left(\sqrt{\frac{\Gamma(2p+1)}{\Gamma^2(p+1)}-1}\right)\frac{\overline{Z}_n^{(p)}}{p\sqrt k}}+\left(\frac{B\left(\ln (n/k)\right)}{\theta}-\frac{p+1}{2\ln(n/k)} 
		\right)(1+{o_{_{\mathbb{P}}}}(1))\right\},\quad
  {p>-1,\quad p\neq 0,}
	\end{equation}
{with $\overline{Z}_n^{(p)}$ an asymptotically standard normal RV.}
	\end{theorem}

\begin{proof} 
When $p=1$, ${\rm PM}_1(k)\equiv H(k)\equiv {\rm H}_0(k)$ and we get for $\widetilde\theta^{\rm GG}_1(k)$ the same distributional representation of the WTC-estimators, $\hat\theta^{\rm GG}(k)$ and {$\hat\theta^{\rm GG}_p(k)$}, for any real $p$. {Considering values of} {$p \in {\mathbb R}$,}
{let $Y_{1:n}, Y_{2:n},\ldots,Y_{n:n}$ be the OS’s associated  with $n$ independent standard Pareto {RV’s}  with {CDF} $F_Y(y)=1-1/y$,\ $y\geq 1$. Standard results from {OS’s} enable to write that  $X_{i:n}\stackrel{d}{=} U(Y_{i:n})$,\ $1\leq i\leq n$ and $Y_{n-i+1:n}/Y_{n-k:n}\stackrel{d}{=}Y_{k-i+1:k}$, where $U(\cdot)$ is the already defined RTQF {in \eqref{eq_quantile_function_U}}.  
From {Caeiro {\it et al.}}\cite{caeiroetal-wtcmop} we know that, if $k$ is intermediate, i.e., if \eqref{intermediate} holds, then the following distributional representation of the log-excesses in \eqref{Vik} holds,
\begin{eqnarray*}
	V_{ik} &\stackrel{d}{=}& \ln U(Y_{n-i+1:n})-\ln
	U(Y_{n-k:n})\\
	&\stackrel{d}{=}&\ln U\left(Y_{n-k:n} Y_{k-i+1:k}\right)-\ln U(Y_{n-k:n}).
\end{eqnarray*}}
{Next, using the relation in \eqref{FOC}
\begin{equation*}
V_{ik} \stackrel{d}{=}
\ln\frac{(\ln (Y_{n-k:n} Y_{k-i+1:k}))^{\theta}}{(\ln (Y_{n-k:n}))^{\theta}}+
\ln\frac{L(\ln (Y_{n-k:n} Y_{k-i+1:k}))}{L(\ln (Y_{n-k:n}))}
\end{equation*}
Given that $E_i=\ln Y_i$} { are independent, identically 
exponentially distributed with mean one, it follows that for $i=1,2,\ldots,k$, and under the validity of the  SOC, in \eqref{SOC},
\begin{eqnarray*}
	V_{ik}& \stackrel{d}{=}&
	\theta\
	\ln\left(1+\frac{E_{k-i+1:k}}{E_{n-k:n}}\right)+B\left(E_{n-k:n}
	\right)
	\frac{\left(1+\frac{E_{k-i+1:k}}{E_{n-k:n}}\right)^\beta-1}{\beta} (1+{o_{_{\mathbb{P}}}}(1)),
	\label{eq_vik_dist}%
\end{eqnarray*}}
{and since $ln(1+x)=x-\tfrac{x^2}{2}+\tfrac{x^3}{3}(1+o(1))$, when $x \to 0$,
 \begin{eqnarray*}
	V_{ik}	& \stackrel{d}{=}&	\theta\, \frac{E_{k-i+1:k}}{E_{n-k:n}}
\left(1-\frac{1}{2}\frac{E_{k-i+1:k}}{E_{n-k:n}}{+\frac{1}{3}\left(\frac{E_{k-i+1:k}}{E_{n-k:n}}\right)^2}(1+{o_{_{\mathbb{P}}}}(1))\right)+B\left(E_{n-k:n}
	\right)
\frac{E_{k-i+1:k}}{E_{n-k:n}} (1+{o_{_{\mathbb{P}}}}(1)).
\end{eqnarray*}
 we can get the asymptotic distributional representation of $V_{ik}^p$,
\begin{eqnarray*}
	V_{ik}^p &\stackrel{d}{=}&
\theta^p\, \left(\frac{E_{k-i+1:k}}{E_{n-k:n}}\right)^p
\left(1-\frac{1}{2}\frac{E_{k-i+1:k}}{E_{n-k:n}}{+\frac{1}{3}\left(\frac{E_{k-i+1:k}}{E_{n-k:n}}\right)^2}(1+{o_{_{\mathbb{P}}}}(1))+\frac{B\left(E_{n-k:n}\right)}{\theta}(1+{o_{_{\mathbb{P}}}}(1))
\right)^p.
\end{eqnarray*}
Using Taylor's expansion  of $(1+x)^p=1+px+\frac{p(p-1)}{2}x^2+o(x^2)$, for {$p>0$ and $|x|\leq 1$ or $p<0$ and $|x|<1$}, and the result $E_{n-k:n} \sim \ln(n/k)$ (see {Girard\cite{girard2004}}), we can write 
\begin{equation*}
V_{ik}^p \stackrel{d}{=}
\theta^p\, \left(\frac{E_{k-i+1:k}}{\ln (n/k)}\right)^p
	\left\{1-\frac{p}{2}\frac{E_{k-i+1:k}}{\ln (n/k)}(1+{o_{_{\mathbb{P}}}}(1))+{\frac{p(3p+5)}{24}}\left(\frac{E_{k-i+1:k}}{\ln (n/k)}\right)^2(1+{o_{_{\mathbb{P}}}}(1)) +\frac{pB\left(\ln (n/k)\right)}{\theta}(1+{o_{_{\mathbb{P}}}}(1))	\right\}.
\end{equation*}
Thus,
\begin{multline}
	\frac{1}{k}\sum_{i=1}^k V_{ik}^p \stackrel{d}{=}
 \frac{\theta^p}{(\ln(n/k))^p}\left\{\frac{1}{k}\sum_{i=1}^k E_i^p -\frac{p}{2\ln(n/k)}\frac{1}{k}\sum_{i=1}^k E_i^{p+1}
 +{\frac{p(3p+5)}{24\ln^2(n/k)}}\frac{1}{k}\sum_{i=1}^k E_i^{p+2}\right\}(1+{o_{_{\mathbb{P}}}}(1))\\
 +\frac{\theta^{p-1} p\, B\left(\ln (n/k)\right)}{(\ln(n/k))^p}
	\frac{1}{k}\sum_{i=1}^k E_i^{p}(1+{o_{_{\mathbb{P}}}}(1)), \label{Vik^p_dist}
\end{multline}}
{and we can write} 
\begin{eqnarray*}
		\frac{1}{k}\sum_{i=1}^k V_{ik}^p &\stackrel{d}{=}&
		\frac{\theta^p\Gamma(p+1)}{(\ln(n/k))^p}\left\{1+\frac{{Z_n^{(p)}(k)}}{\Gamma(p+1)\sqrt k}-\left(\frac{p(p+1)}{2\ln(n/k)}- \frac{p\, B\left(\ln (n/k)\right)}{\theta}
		\right)(1+{o_{_{\mathbb{P}}}}(1))\right\},
\end{eqnarray*}
{where $Z_n^{(\alpha)}(k):= \sqrt k \left(\frac{1}{k} \sum_{i=1}^k E_i^\alpha - \Gamma(\alpha+1)\right)$ is an {asymptotically  normal RV with mean 0 and variance $\Gamma(2\alpha+1)-\Gamma^2(\alpha+1)$} and $\frac{1}{k}\sum_{i=1}^k E_i^\alpha \overset {p} {\to} \Gamma(\alpha+1)$, for {$\alpha\geq 0$}. Then, 
considering {$p>-1$ and $p\neq 0$,} 
the following distributional representation holds, 
}
\begin{eqnarray*}
{\rm PM}_p(k) &\stackrel{d}{=} & \frac{\theta}{\ln(n/k)}\left\{1+\frac{{Z_n^{(p)}(k)}}{\Gamma(p+1)\sqrt k}-\left(\frac{p(p+1)}{2\ln(n/k)}-\frac{p\, B\left(\ln (n/k)\right)}{\theta}
\right)(1+{o_{_{\mathbb{P}}}}(1))\right\}^{1/p}\\
&\stackrel{d}{=} &\frac{\theta}{\ln(n/k)}\left\{1+{\left(\sqrt{\frac{\Gamma(2p+1)}{\Gamma^2(p+1)}-1}\right)\frac{\overline{Z}_n^{(p)}}{p\sqrt k}}-\left(\frac{p+1}{2\ln(n/k)}- \frac{B\left(\ln (n/k)\right)}{\theta}
\right)(1+{o_{_{\mathbb{P}}}}(1))\right\},\\
\end{eqnarray*}
{being $\overline{Z}_n^{(p)}$ an asymptotically standard normal RV.}
With $\widetilde\theta^{\rm GG}_p(k):=\ln(n/k) {\rm PM}_p(k)$, as in \eqref{WTC-PM-GG}, the asymptotic distributional representation in \eqref{WTC-MP-GG_dist} holds. 
\end{proof}

Next, we {summarize} 
the asymptotic properties of the WTC-estimator $\widetilde \theta^{\rm G}_p(k)$, in \eqref{WTC-PM-G}. 

\begin{theorem}\label{Theo-wtcpmG}
	Under a second-order framework, i.e., assuming that condition  \eqref{SOC} holds, with $k$ an intermediate sequence, as in \eqref{intermediate}, the WTC-estimator $\widetilde\theta^{G}_p(k)$, in \eqref{WTC-PM-G}, has an asymptotic distributional representation of the type,
\begin{equation}\label{WTC-MP-G_dist}
		\widetilde\theta^{\rm G}_p(k)\stackrel{d}{=} \theta \left\{1+{\left(\sqrt{\frac{\Gamma(2p+1)}{\Gamma^2(p+1)}-1}\right)\frac{\overline{Z}_n^{(p)}}{p\sqrt k}}+\left(\frac{B\left(\ln (n/k)\right)}{\theta}-\frac{p-1}{2\ln(n/k)} 
		\right)(1+{o_{_{\mathbb{P}}}}(1))\right\},\quad
  {p>-1,\quad p\neq 0}.
	\end{equation}
	\end{theorem}

\begin{proof}
The proof follows straightforwardly from the asymptotic relation in equation \eqref{G-GG}.    
\end{proof}

\begin{corollary}\label{cor-wtcpm}
      Under the conditions of Theorems \ref{Theo-wtcpmGG} and \ref{Theo-wtcpmG}, and further assuming that 
 $\sqrt k B(\ln(n/k)) \to \lambda \in \mathbb{R}$ and $\sqrt k/ \ln(n/k) \to 0$, and with $\widetilde\theta^{\bullet}_p(k)$, {with $p>-1$, $p\neq 0$}, and $\bullet=\{GG,G\}$ generally denoting the estimators in \eqref{WTC-PM-GG} and \eqref{WTC-PM-G},
\begin{equation*}
	\sqrt k \left(\widetilde\theta^{\bullet}_p(k) -\theta\right) \xrightarrow[]{d} \mathcal{N}\left(\lambda, \left({\frac{\theta}{p}\sqrt{\frac{\Gamma(2p+1)}{\Gamma^2(p+1)}-1}}\right)^2\right).
\end{equation*}
\end{corollary}

\begin{proof}
From the distributional representations in \eqref{WTC-MP-GG_dist} and \eqref{WTC-MP-G_dist}, 
\begin{equation*}
\sqrt k \left(\widetilde\theta^{\rm GG}_p(k) -\theta \right)\stackrel{d}{=} {\left(\sqrt{\frac{\Gamma(2p+1)}{\Gamma^2(p+1)}-1}\right){\frac{\theta \overline{Z}_n^{(p)}}{p}}}  +\left(\sqrt k B\left(\ln (n/k)\right)- \frac{\theta \sqrt k(p+1)}{2\ln(n/k)} \right)(1+{o_{_{\mathbb{P}}}}(1)),
\end{equation*}
and 
\begin{equation*}
\sqrt k \left(\widetilde\theta^{\rm G}_p(k) -\theta \right)\stackrel{d}{=} {\left(\sqrt{\frac{\Gamma(2p+1)}{\Gamma^2(p+1)}-1}\right){\frac{\theta \overline{Z}_n^{(p)}}{p}}}  +\left(\sqrt k B\left(\ln (n/k)\right)- \frac{\theta \sqrt k(p-1)}{2\ln(n/k)} 
\right)(1+{o_{_{\mathbb{P}}}}(1)).
\end{equation*}
Assuming that $\sqrt k B(\ln(n/k)) \to \lambda \in \mathbb{R}$ and $\sqrt k/ \ln(n/k) \to 0$, with {$\overline{Z}_n^{(p)}$} the standard normal  {RV} in \eqref{WTC-MP-GG_dist} and \eqref{WTC-MP-G_dist}, 
the result follows. 
\end{proof}

Next, we present without proof the asymptotic behaviour of $\widetilde\theta^{\rm GG}_p(k)$ and $\widetilde\theta^{\rm G}_p(k)$, in \eqref{WTC-PM-GG} and \eqref{WTC-PM-G}, for models with a function $B(t)$ given by $B(t)=\alpha\, t^{-1}$, $\alpha \in \mathbb{R}$.

\begin{corollary}\label{cor-wtcpm-new}
  Under the conditions of Theorems \ref{Theo-wtcpmGG} and \ref{Theo-wtcpmG}, 
    and further {assuming that}
  $B(t)=\alpha\, t^{-1}$, with $\alpha \in \mathbb{R}$, the WTC-estimators $\widetilde\theta^{\rm GG}_p(k)$ and $\widetilde\theta^{\rm G}_p(k)$\, 
  {with $p>-1$, and $p\neq 0$}, have an asymptotic distributional representation of the type,
	\begin{equation}\label{WTC-PM-GG-dist1}
    \widetilde\theta^{\rm GG}_p(k)\stackrel{d}{=} \theta \left\{1+{\left(\sqrt{\frac{\Gamma(2p+1)}{\Gamma^2(p+1)}-1}\right)\frac{\overline{Z}_n^{(p)}}{p\sqrt k}}+\frac{2\alpha -\theta(p+1)}{2 \theta \ln(n/k)} 
		(1+{o_{_{\mathbb{P}}}}(1))\right\},
    \end{equation}
 and 
 \begin{equation}\label{WTC-PM-G-dist1}
    \widetilde\theta^{\rm G}_p(k)\stackrel{d}{=} \theta \left\{1+{\left(\sqrt{\frac{\Gamma(2p+1)}{\Gamma^2(p+1)}-1}\right)\frac{\overline{Z}_n^{(p)}}{p\sqrt k}}+\frac{2\alpha-\theta(p-1)}{2 \theta \ln(n/k)} 
		(1+{o_{_{\mathbb{P}}}}(1))\right\}.
    \end{equation}
If the intermediate sequence $k$ is such that {$\sqrt k B(\ln(n/k)) \to \lambda \in \mathbb{R}$}, and with $\widetilde\theta^{\bullet}_p(k)$, $\bullet=\{GG,G\}$, then
		\begin{equation*}
			\sqrt k \left(\widetilde\theta^{\bullet}_p(k) -\theta\right) \xrightarrow[]{d} \mathcal{N}\left({\lambda\; b^\bullet(\theta,\alpha,p)}, \left({\frac{\theta}{p}\sqrt{\frac{\Gamma(2p+1)}{\Gamma^2(p+1)}-1}}\right)^2\right),
		\end{equation*}
  {with 
  $$b^\bullet(\theta,\alpha,p):=\left\{
  \begin{array}{cc}
  \frac{2\alpha -\theta(p+1)}{2\alpha}, & \bullet=GG,\\
  \frac{2\alpha -\theta(p-1)}{2\alpha}, & \bullet=G.
 \end{array}
  \right.$$}
\end{corollary}

The results presented in \eqref{WTC-PM-GG-dist1} and \eqref{WTC-PM-G-dist1} enable the determination of the optimal $p$-values, i.e., the values of $p$ such that the dominant component of the bias is null. So, for the WTC-estimators $\widetilde\theta^{\rm GG}_p(k)$ and $\widetilde\theta^{\rm G}_p(k)$, the optimal $p$-values, denoted by $p^{\rm GG}_{opt}$ and $p^{\rm G}_{opt}$ are:
\begin{equation*}
p^{\rm GG}_{opt}=\left\{p: p=\frac{2\alpha}{\theta}-1 \land p>0, \alpha \in \mathbb{R}, \theta>0\right\},
\end{equation*}
and
\begin{equation*}
p^{\rm G}_{opt}=\left\{p: p=\frac{2\alpha}{\theta}+1 \land p>0, \alpha \in \mathbb{R}, \theta>0\right\}.
\end{equation*}

\begin{remark}
For the Logistic model with  {CDF, $F(x)=1-\frac{1}{1+\exp(x)}$}, we have $\theta=1$ and $B(t)=-\ln 2/t$ ($\alpha=-\ln 2$), and so, is not possible to find optimal $p$-values for  $\widetilde\theta^{\rm GG}_p(k)$ and  $\widetilde\theta^{\rm G}_p(k)$. Considering the Gumbel model with CDF, $F(x)=\exp(-\exp(-(x-\mu)))$, $\mu \in \mathbb{R}$, we have $\theta=1$ and $B(t)=-\mu/t$ ($\alpha=-\mu$). For this model, $p^{GG}_{opt}$ can be computed for values of $\mu <-0.5$ (for instance, when $\mu=-1$, $p^{\rm GG}_{opt}=1$) and $p^{\rm G}_{opt}$ can be computed for values of $\mu <0.5$ (for instance, when $\mu=-1$, $p^{\rm G}_{opt}=3$). 
\end{remark}

\begin{corollary}\label{cor-wtcpm-new2}
  Under the conditions of Theorems \ref{Theo-wtcpmGG} and \ref{Theo-wtcpmG}, 
    and further assuming that
 $B(t)=0$, the WTC-estimators $\widetilde\theta^{\rm GG}_p(k)$ and $\widetilde\theta^{\rm G}_p(k)$ {with $p>-1$, and $p\neq 0$}, have an asymptotic distributional representation of the type,
	\begin{equation*}\label{WTC-PM-GG-dist0}
    \widetilde\theta^{\rm GG}_p(k)\stackrel{d}{=} \theta \left\{1+{\left(\sqrt{\frac{\Gamma(2p+1)}{\Gamma^2(p+1)}-1}\right)\frac{\overline{Z}_n^{(p)}}{p\sqrt k}}-\frac{p+1}{2\ln(n/k)} 
		(1+{o_{_{\mathbb{P}}}}(1))\right\},
    \end{equation*}
 and 
 \begin{equation*}\label{WTC-PM-G-dist0}
    \widetilde\theta^{\rm G}_p(k)\stackrel{d}{=} \theta \left\{1+{\left(\sqrt{\frac{\Gamma(2p+1)}{\Gamma^2(p+1)}-1}\right)\frac{\overline{Z}_n^{(p)}}{p\sqrt k}}-\frac{p-1}{2 \ln(n/k)} 
		(1+{o_{_{\mathbb{P}}}}(1))\right\}.
    \end{equation*}
If the intermediate sequence $k$ is such that{$\sqrt k /\ln(n/k)) \to \eta \in \mathbb{R}$}, and with $\widetilde\theta^{\bullet}_p(k)$, $\bullet=\{GG,G\}$, then
		\begin{equation*}
			\sqrt k \left(\widetilde\theta^{\bullet}_p(k) -\theta\right) \xrightarrow[]{d} \mathcal{N}\left(\eta\; b_0^\bullet(\theta, p), \left({\frac{\theta}{p}\sqrt{\frac{\Gamma(2p+1)}{\Gamma^2(p+1)}-1}}\right)^2\right),
		\end{equation*}
with 
  $$b_0^\bullet(\theta,p):=\left\{
  \begin{array}{cc}
  -\frac{\theta(p+1)}{2}, & \bullet=GG\\
  -\frac{\theta(p-1)}{2}, & \bullet=G.
 \end{array}
  \right.$$
\end{corollary}

\begin{remark}\label{remark_wtcpmG}%
From the limit distribution presented in Corollary \ref{cor-wtcpm-new2}, it can be inferred that  $\widetilde\theta^{\rm G}_p(k)$ is asymptotically unbiased under the conditions of
$B(t)=0$ and 
$p=1$. This specific scenario holds true for the Exponential and Weibull distributions since both models have $B(t)=0$. Hence one expects in practice that the values of $\widetilde\theta^{\rm G}_p(k)$ will closely align with the true value of $\theta$  within a certain range of $k$-values. The extent of this range may differ, being either more extensive or more limited, depending on the relevance of the remaining higher-order bias terms.    
\end{remark}

\subsection{Asymptotic {MSE} of the estimators}

In this subsection we present the \textit{asymptotic} MSE (AMSE) of the WTC-estimators under study, defined in \eqref{WTC-Mop-GG}, \eqref{WTC-Mop-G}, \eqref{WTC-PM-GG} and \eqref{WTC-PM-G}, for models where the function $B(t)$, in \eqref{SOC}, can assume 
a generic expression, models with $B(t)=\alpha t^{-1}$, as the Logistic and Gumbel distributions, and models where $B(t)$ can be considered null, i.e., $B(t)=0$, as the Weibull and the standard Exponential models. With $\bar{\theta}_{p}^{\bullet}(k)$ generically denoting any of the aforementioned WTC-estimators, we can write
$$ {\rm AMSE(\bar{\theta}_{p}^{\bullet}(k))}=\left({\rm Bias}_\infty(\bar{\theta}_{p}^{\bullet}(k))\right)^2 + {\rm Var}_\infty(\bar{\theta}_{p}^{\bullet}(k)),$$
where ${\rm Bias}_\infty(\bar{\theta}_{p}^{\bullet}(k))$ denotes the asymptotic bias and ${\rm Var}_\infty(\bar{\theta}_{p}^{\bullet}(k))$ the asymptotic variance, being the latter given by the expression 
$${\rm Var}_\infty(\bar{\theta}_{p}^{\bullet}(k))=\frac{\theta^2}{k}v(p)=\frac{\theta^2}{k}\left\{ 
\begin{array}{cc}
1, & \bar{\theta}_{p}^{\bullet}(k)={\widehat \theta}_{p}^{\bullet}(k),\\[6pt]
\frac{1}{p^2}\left(\frac{\Gamma(2p+1)}{\Gamma^2(p+1)}-1\right), & \bar{\theta}_{p}^{\bullet}(k)={\widetilde \theta}_{p}^{\bullet}(k). 
\end{array}
\right.
$$
Therefore ${\rm Var}_\infty(\widetilde{\theta}_{p}^{\bullet}(k))\geq {\rm Var}_\infty(\widehat{\theta}_{p}^{\bullet}(k))$,  $\bullet=\{GG, G\}$, for values of $p>0$, being the equality achieved when $p=1$.

 The limiting representations presented in Theorems \ref{Theo-wtcmopGG} and \ref{Theo-wtcmopG}, and the results in Corollary \ref{cor_wtcmop} enable us to write the AMSE of the WTC-estimators in \eqref{WTC-Mop-GG}, \eqref{WTC-Mop-G}, generally denoted by $\widehat\theta^{\bullet}_p$, with $\bullet=\{GG,G\}$,  where $B(t)$ represents a generic function, by the following expression
\begin{equation*}\label{AMSE_wtcmop_geral}
{\rm AMSE}(\widehat \theta^{\bullet}_p)=\left(\frac{B(\ln(n/k))}{\theta} \right)^2+ \frac{\theta^2}{k}.
\end{equation*}
In the same lines, considering also models with a generic $B(t)$ function, from Theorems \ref{Theo-wtcpmGG} and \ref{Theo-wtcpmG} and the results in Corollary \ref{cor-wtcpm}, with $\widetilde\theta^{\bullet}_p$, $\bullet=\{GG,G\}$, generally denoting the WTC-estimators in \eqref{WTC-PM-GG} and \eqref{WTC-PM-G}, we get
 \begin{equation*}\label{AMSE_wtcpm_geral}
{\rm AMSE}(\widetilde \theta^{\bullet}_p)= 
   \left(\frac{B(\ln(n/k))}{\theta}\right)^2+  \frac{\theta^2}{k\,p^2}\left(\frac{\Gamma(2p+1)}{\Gamma^2(p+1)}-1\right).
\end{equation*}
For values of $p=1$, the ${\rm AMSE}(\widehat \theta^{\bullet}_1)={\rm AMSE}(\widetilde \theta^{\bullet}_1)$, $\bullet=\{GG,G\}$.

We now establish the AMSE of the models with function $B(t)=\alpha t^{-1}$, i.e., $B(\ln(n/k))=\alpha/\ln(n/k)$, $\alpha \in \mathbb{R}$, based on the results in  Corollaries \ref{cor-wtcmopGG-new} and \ref{cor-wtcmopG-new} for the WTC-estimators $\widehat \theta^{\bullet}_p$, and the results in Corollary \ref{cor-wtcpm-new} for the WTC-estimators $\widetilde \theta^{\bullet}_p$, $\bullet=\{GG,G\}$, 
\begin{equation}\label{AMSE_wtcmop1}
{\rm AMSE}(\widehat \theta^{\bullet}_p)=\left\{
\begin{array}{cc} 
   \left(\frac{\alpha-\theta}{\ln(n/k)} \right)^2+ \frac{\theta^2}{k},  &  \bullet=GG,\\[6pt]
  \left(\frac{\alpha}{\ln(n/k)} \right)^2+ \frac{\theta^2}{k},   &  \bullet=G,
    \end{array}
     \right.
\end{equation}
and 
\begin{equation}\label{AMSE_wtcpm1}
{\rm AMSE}(\widetilde \theta^{\bullet}_p)=\left\{
\begin{array}{cc} 
   \left(\frac{2\alpha-{\theta(p+1)}}{2 \ln(n/k)}  \right)^2+ \frac{\theta^2}{k\,p^2}\left(\frac{\Gamma(2p+1)}{\Gamma^2(p+1)}-1\right),  &  \bullet=GG,\\[6pt]
    \left(\frac{2\alpha-{\theta(p-1)}}{2 \ln(n/k)}  \right)^2+ \frac{\theta^2}{k\,p^2}\left(\frac{\Gamma(2p+1)}{\Gamma^2(p+1)}-1\right),  &  \bullet=G.
    \end{array}
     \right.
\end{equation}

\noindent
From \eqref{AMSE_wtcmop1}, we conclude that if $\theta =2\alpha$, the two estimators have the same AMSE and that AMSE($\widehat \theta^{\rm GG}_p$)<AMSE($\widehat \theta^{\rm G}_p$) when $\theta<2\alpha$. Since $\theta>0$ we need to have $\alpha >0$ in order to achieve this relation between the two AMSE’s. Similarly, from \eqref{AMSE_wtcpm1}, we can conclude that if $\theta=\frac{2\alpha}{p}$ the two estimators have the same AMSE, and that AMSE($\widetilde \theta^{\rm GG}_p$)<AMSE($\widetilde \theta^{\rm G}_p$) when $\theta<\frac{2\alpha}{p}$. Since $\theta>0$ the only way to obtain {such a relation} is when $\alpha$ and $p$ are both positive.

If we consider the results presented in Corollaries \ref{cor-wtcmopGG-new} and \ref{cor-wtcmopG-new} for the WTC-estimators $\widehat \theta^{\bullet}_p$, $\bullet=\{GG,G\}$, for models with $B(t)=0$ and $\theta > 0$ then we get the following AMSE,  
\begin{equation*}\label{AMSE_wtcmop0}
{\rm AMSE}(\widehat \theta^{\bullet}_p)=\left\{
\begin{array}{cc}
    \frac{\theta^2}{\ln^2(n/k)}+  \frac{\theta^2}{k} , & \bullet=GG,\\[6pt]
    \frac{\theta^2}{k} , & \bullet=G,
\end{array}
\right. 
\end{equation*}    
and from Corollary \ref{cor-wtcpm-new2} we obtain that 
\begin{equation*}\label{AMSE_wtcpm0}
{\rm AMSE}(\widetilde \theta^{\bullet}_p)= \left\{
\begin{array}{cc}
\left(\frac{\theta(p+1)}{2\ln(n/k)}\right)^2+\frac{\theta^2}{k\,p^2}\left(\frac{\Gamma(2p+1)}{\Gamma^2(p+1)}-1\right), & \bullet=GG,\\[6pt]
\left(\frac{\theta(p-1)}{2\ln(n/k)}\right)^2+\frac{\theta^2}{k\,p^2}\left(\frac{\Gamma(2p+1)}{\Gamma^2(p+1)}-1\right), & \bullet=G.\\
\end{array}
\right. 
\end{equation*}

\vspace{5pt}%
\noindent For the latter AMSE, we can also conclude that  AMSE($\widetilde \theta^{\rm G}_p$)<AMSE($\widetilde \theta^{\rm GG}_p$), for every $p>0$.

{
\begin{remark}
The results presented allowed for a comparison between WTC-estimators within the same class. However, comparing the AMSE of the estimators $\widehat\theta_p^{\bullet}(k)$ in \eqref{WTC-Mop-GG} and \eqref{WTC-Mop-G}, and $\widetilde\theta_p^{\bullet}(k)$ in \eqref{WTC-PM-GG} and \eqref{WTC-PM-G}—i.e., between the two classes—is more complex, given that each class of estimators has different asymptotic bias and asymptotic variances. An analysis similar to the one presented in \cite{HP1998}, among other papers, for the extreme value index, where an asymptotic comparison at optimal levels was performed could be put forward after some extra research. First, we would need to determine the optimal levels of the WTC-estimators under study, i.e., the values $k$ that minimize the AMSE of the WTC-estimators. Under a regular variation context, we would then define the limiting mean square error and a measure of asymptotic relative efficiency to compare WTC-estimators. 
\end{remark}
}

\vskip 1cm 
\section{Finite sample properties of the WTC-estimators}
We conducted an extensive Monte-Carlo simulation study  to assess the performance of the aforementioned WTC-estimators $\widehat\theta^{\rm GG}_p(k)$, $\widehat\theta^{\rm G}_p(k)$, $\widetilde\theta^{\rm GG}_p(k)$ and  $\widetilde\theta^{\rm G}_p(k)$   in 
\eqref{WTC-Mop-GG},  \eqref{WTC-Mop-G},  \eqref{WTC-PM-GG} and \eqref{WTC-PM-G}, respectively. %
The parameter $p$ values were chosen based on a preliminary simulation study. The value $p =1$ was consistently employed in $\widetilde\theta^{\rm G}_p(k)$ to facilitate a comparison with an established estimator from the existing literature.  This study has been carried out for simulated datasets from standard Exponential ($\theta=1$), Weibull(2,1) ($\theta=0.5$), Gamma(0.75,1) ($\theta=1$), Half-Normal ($\theta=0.5$), standard Gumbel ($\theta=1$) and Half-Logistic ($\theta=1$) distributions, under different sample sizes. The distributions here considered are partially provided in {Gardes and Girard} \cite{gardes2008estimation}, Table 1. Remark that in case of the Gumbel distribution, only the positive sample values are used to compute the estimates. %
All the computations were done using R software. %
Let $\bar{\theta}_{p}^{\bullet}(k)$ generically denote any of the aforementioned WTC-estimators. 
For each distribution, sample size and value of the parameter $p$, we generated 5000 samples and computed for the $i$-th sample the values of 
\begin{equation*}
\bar{\theta}_{p,i}^{\bullet}(k),\quad  k=1,\,2, \ldots, n-1,\quad i=1,\,2, \ldots, 5000.    
\end{equation*}
Next we obtained the Monte-Carlo estimates of the mean value (E)  
\begin{align}
{\rm E}[\bar{\theta}_{p}(k)]=\sum_{i=1}^{5000} \frac{\bar{\theta}_{p,i}(k)}{5000}, 
\qquad 1\leq k\leq n-1,
\label{simulE_E}%
\end{align}
and the RMSE,
\begin{align}
{\rm RMSE}[\bar{\theta}_{p}(k)]=\sqrt{\sum_{i=1}^{5000}\frac{(\bar{\theta}_{p,i}(k)-\theta)^2}{5000}},\qquad 1\leq k\leq n-1.
\label{simulE_RMSE}%
\end{align}
In addition, we have further computed the simulated  optimum level,
\begin{equation}
\hat{k}_{p,0}= 
\arg \min_{k} {\rm RMSE}[\bar{\theta}_{p}(k)],
\label{eq:osf}%
\end{equation}
and obtained
\begin{align}
{\rm E}[\bar{\theta}_p(\hat{k}_{p,0})] 
\quad \text{and}\quad 
{\rm RMSE}[\bar{\theta}_p(\hat{k}_{p,0})].
\label{simulE_RMSE_k0}%
\end{align}
The optimal level $k$ in \eqref{eq:osf} is not highly relevant in practice, but serves as an indicator of the optimal potential performance for each estimator, 
possibly achieved in practice through a suitable adaptive estimation approach.
However, the adaptive choice of an optimal value of $k$ is still an open problem. Additional details can be found in the references 
\cite{caballero2017technical,caeiro2016threshold,just2021theappropriate}.
Note that the simulated optimal sample fraction $\hat{k}_{p,0}/n$ is provided only in appendix.

\subsection{Mean values and {RMSE} patterns as $k$-functionals}\label{sec_mean_value_rmse_patterns}
In Figures \ref{fig:exp}--\ref{fig:hl}, we present the Monte-Carlo estimates of the mean value (left), and the RMSE (right),
with respect to $k$, for the four different classes of WTC-estimators.
The dashed line in the left plots indicates the true value of the parameter $\theta$. The performance of the estimators is evaluated based on the smoothness of the simulated mean value curve, close to the true value of $\theta$ and by a small RMSE in such a region. However, a low RMSE alone, in the absence of a stable region for the mean value estimates, might not be relevant in practical  applications. Based on these figures, we can draw the following conclusions:
\begin{itemize}
\item
The new estimator $\widetilde \theta^{\rm G}_p$ with an adequate choice of $p$ outperforms all other estimators in terms of bias and RMSE for almost all models considered in the paper. Such behaviour is achieved by employing $p=1$  for the Exponential and Weibull models, $p=2$ for the Gamma model, $p=3.5$ for the Half-Normal model and $p=0.25$ for the Gumbel and Half-Logistic models.
\item In practical applications, it is highly recommended to consider a range of values for the parameter $p$ and generate the associated sample paths of the estimates against $k$.
Next, the most suitable $p$ value can be determined by identifying the one which provides the highest stability for a larger range of k$-$values. References \cite{ghrfam2013,penalva2020lehmer,gomesetal2020reduced-bias} provide heuristic algorithms based on
sample path stability to select the tuning parameter $p$ and the threshold $k$.

\item As expected, for the standard Exponential and Weibull distributions, $\widetilde \theta^{\rm G}_1$ has near zero-bias for almost all $k$-values (see remark \ref{remark_wtcpmG}).
\item As can be noticed in  figure \ref{fig:gumbel}, the estimator $\widehat \theta^{\rm G}_{-2}$ behaves better with respect to the minimal RMSE. However $\widetilde \theta^{\rm G}_{0.25}$ provides a similar RMSE over an extended range of $k$-values.

\item For the Half-Logistic distribution, the smallest RMSE is provided by the $\widehat\theta^{\rm GG}_1(k)$ estimator. However such value is only achieved due to a change of the bias sign, as the value of $k$ increases. Thus  $\widehat\theta^{\rm GG}_1(k)$ is not relevant for this distribution, since the smallest RMSE is not associated to a stable region for the mean value estimates.
\end{itemize}

\begin{figure}[h!]
\includegraphics[width=0.45\textwidth]{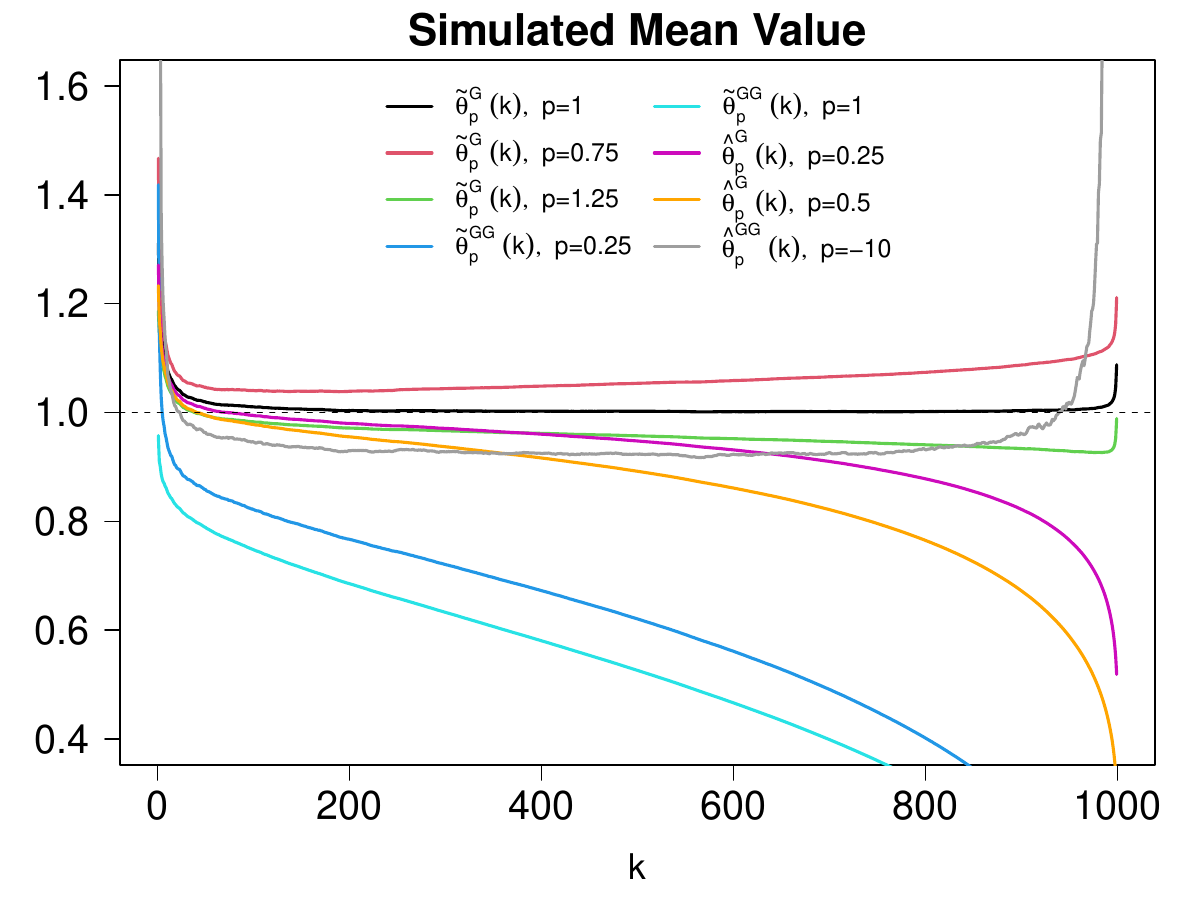}\includegraphics[width=0.45\textwidth]{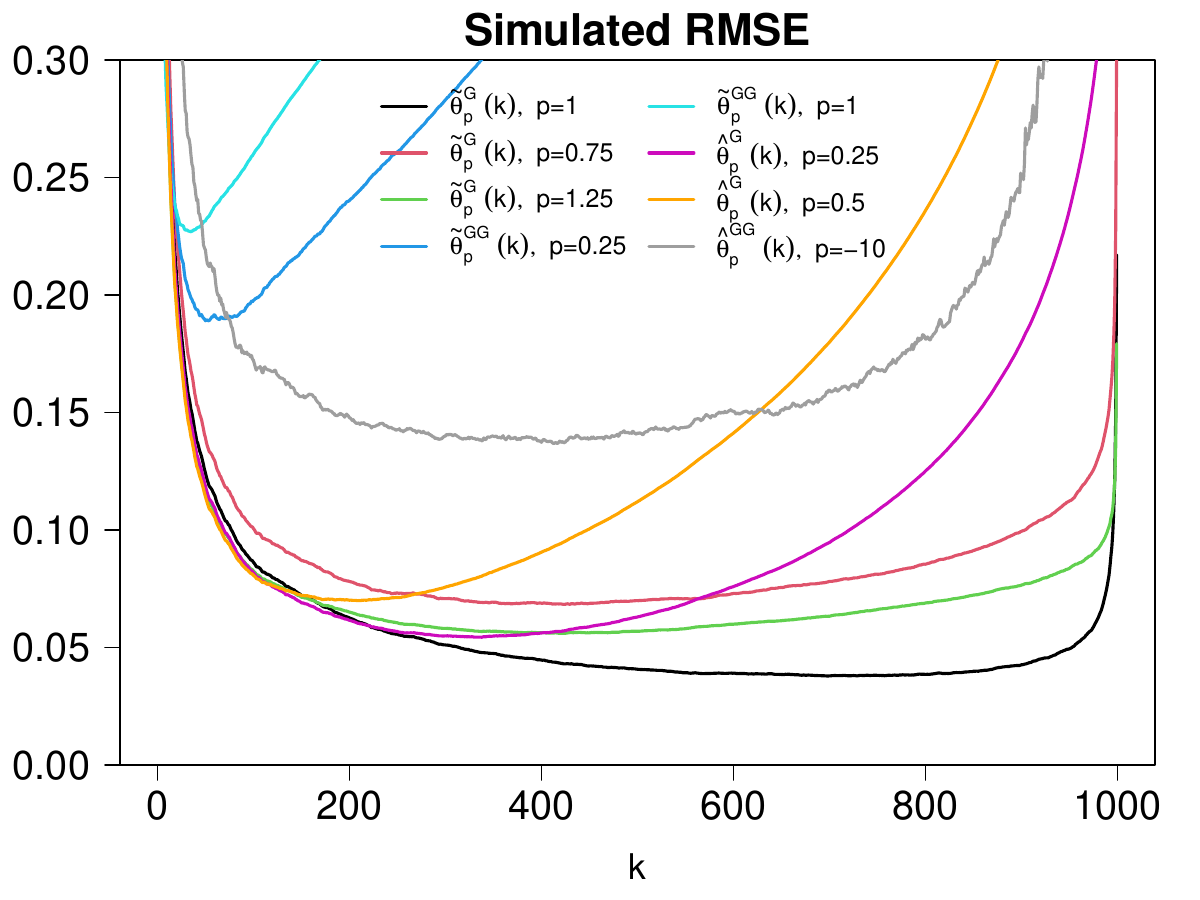}
\caption{Simulated mean values (left) and RMSE (right) for the Exponential model.}
\label{fig:exp}%
\end{figure}

\begin{figure}[h!]
\includegraphics[width=0.45\textwidth]{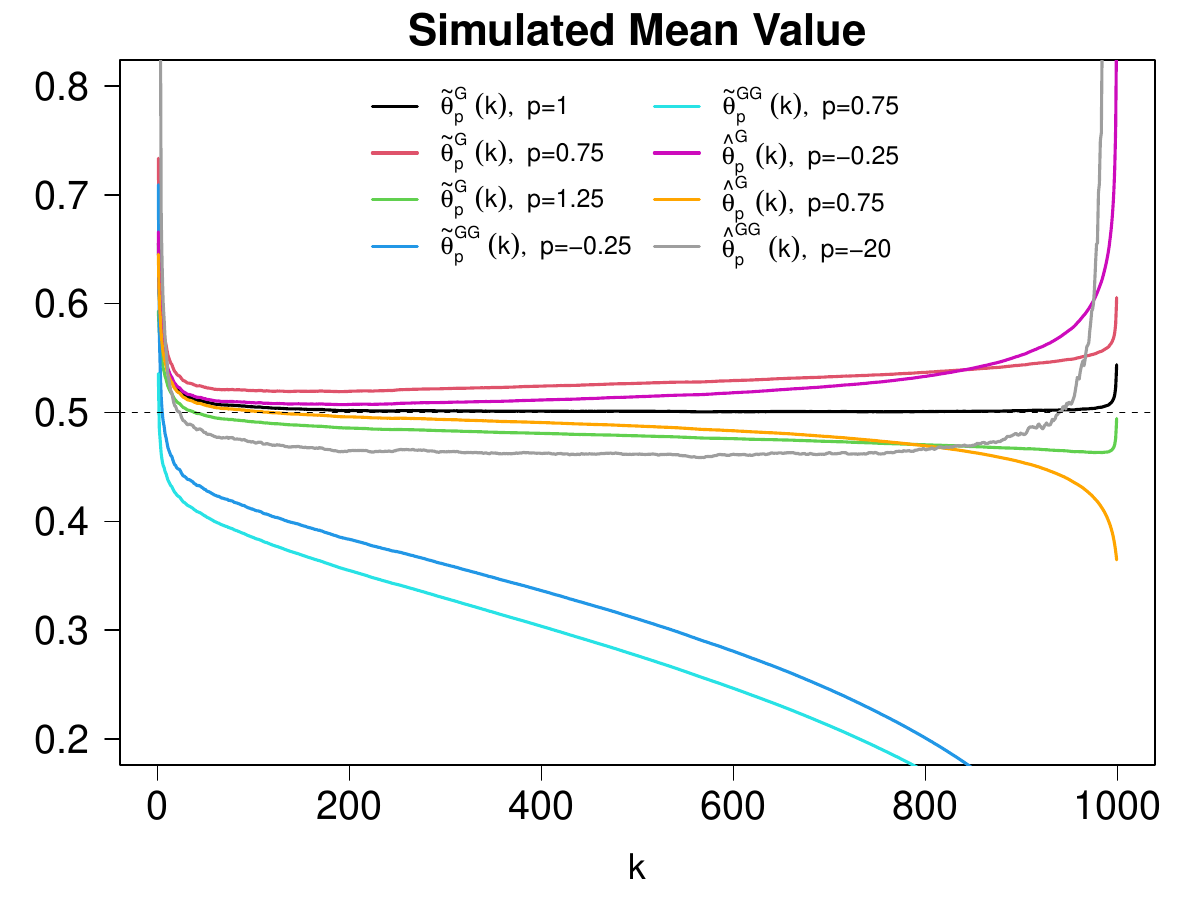}\includegraphics[width=0.45\textwidth]{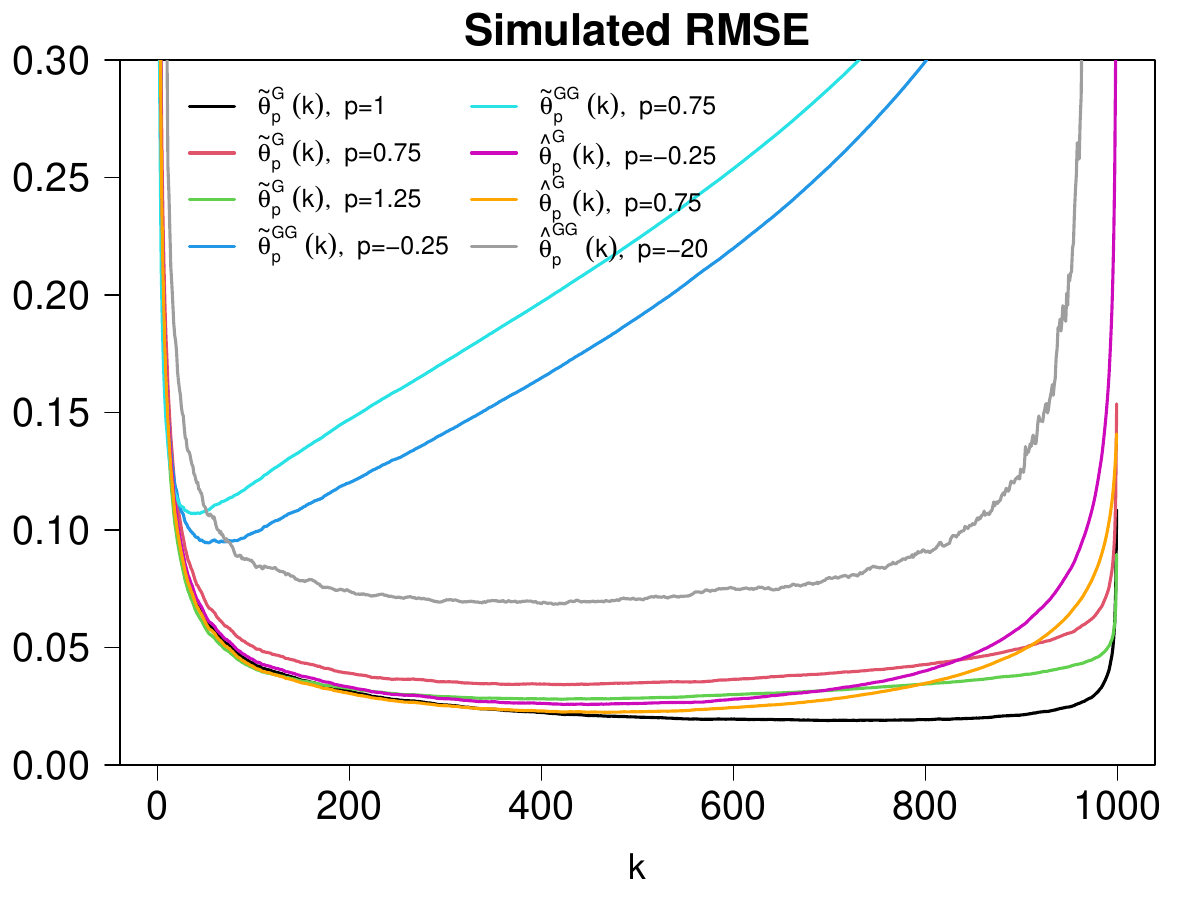}
\caption{Simulated mean values (left) and RMSE (right) for the Weibull(2,1) model.}
\end{figure}

\begin{figure}[h!]
\includegraphics[width=0.45\textwidth]{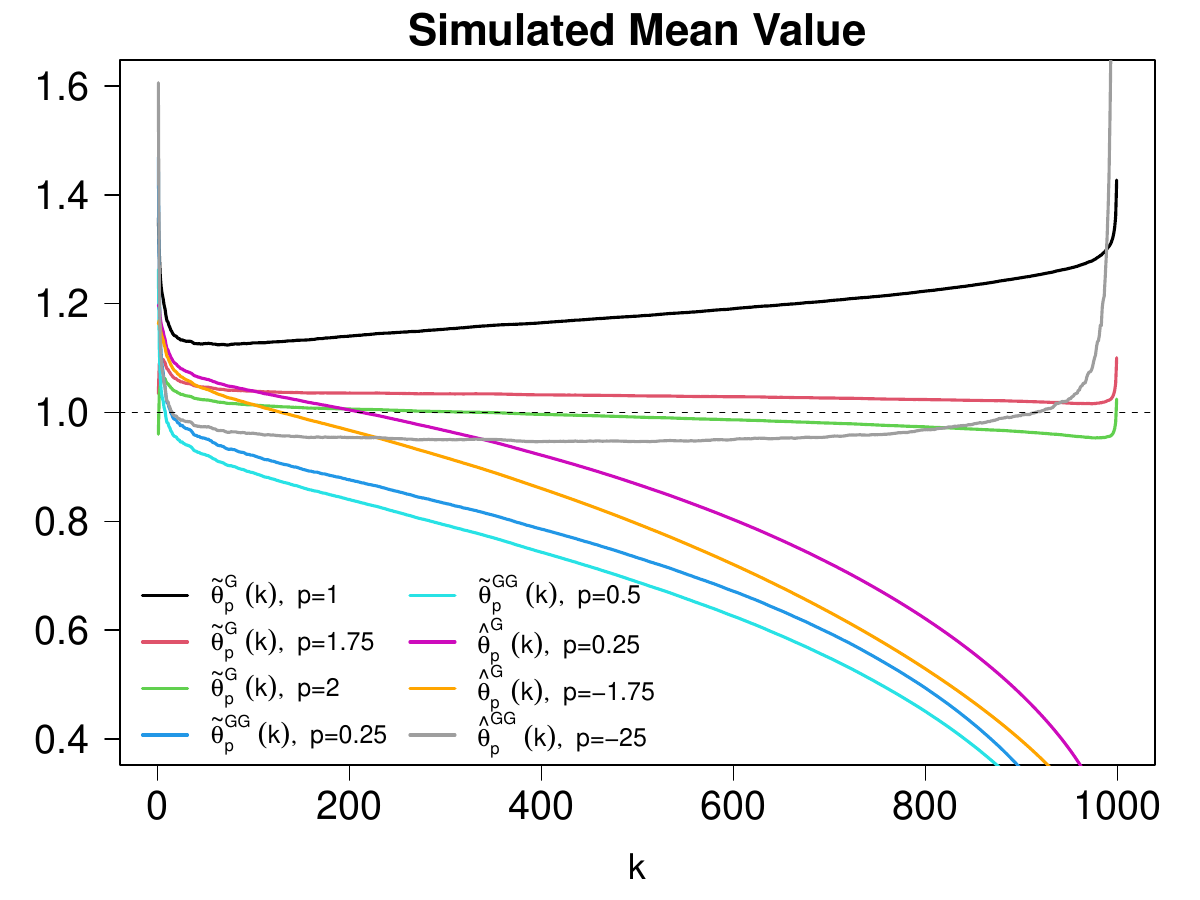}\includegraphics[width=0.45\textwidth]{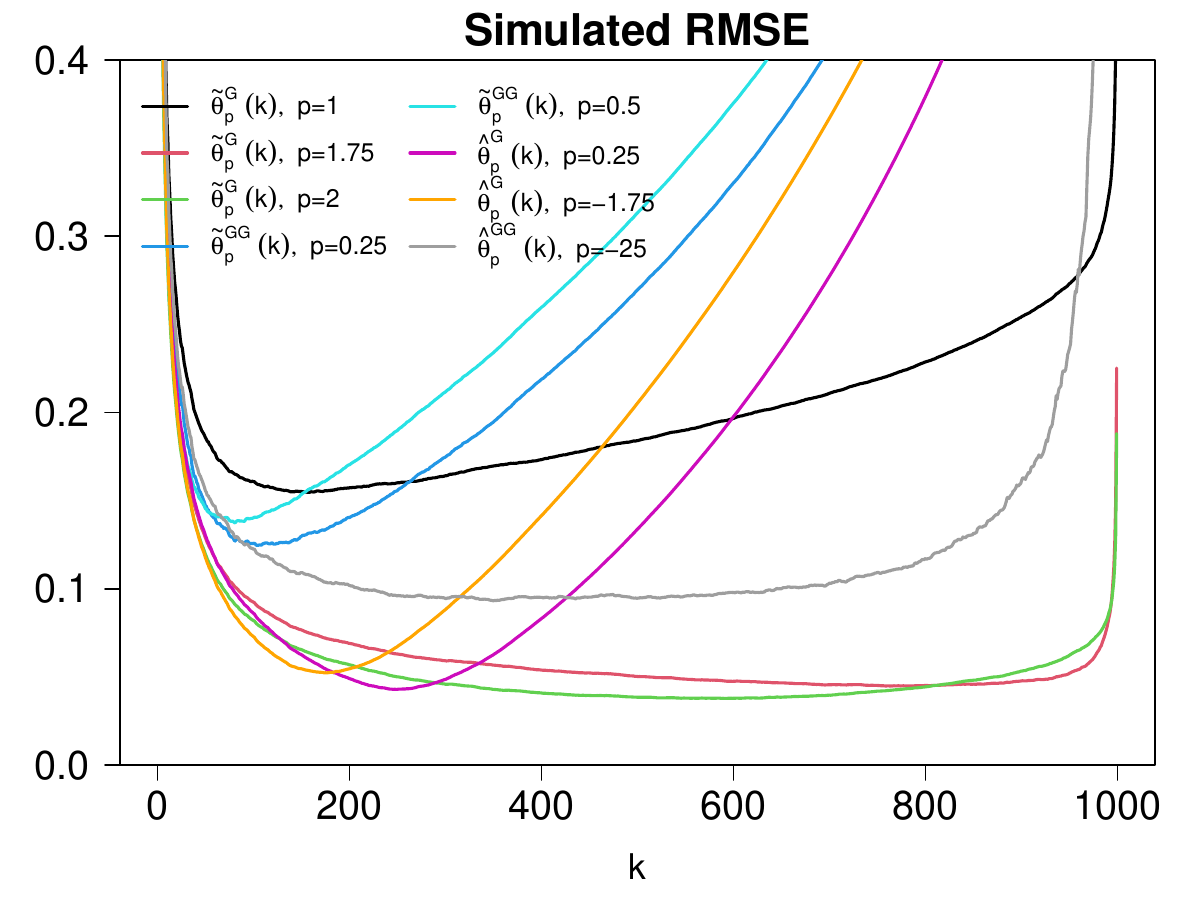}
\caption{Simulated mean values (left) and RMSE (right) for the Gamma(0.75,1) model.}
\label{fig:gamma075}%
\end{figure}

\begin{figure}[h!]
\includegraphics[width=0.45\textwidth]{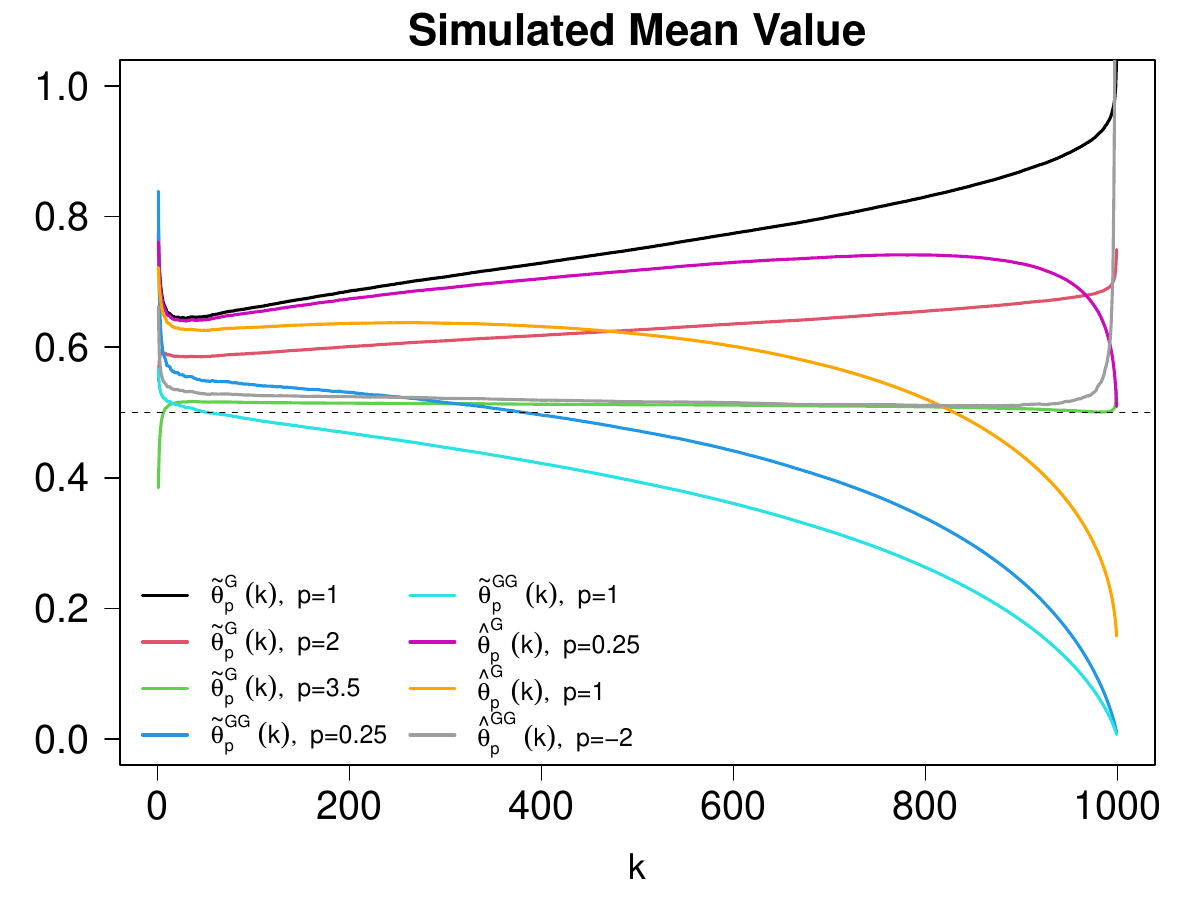}\includegraphics[width=0.45\textwidth]{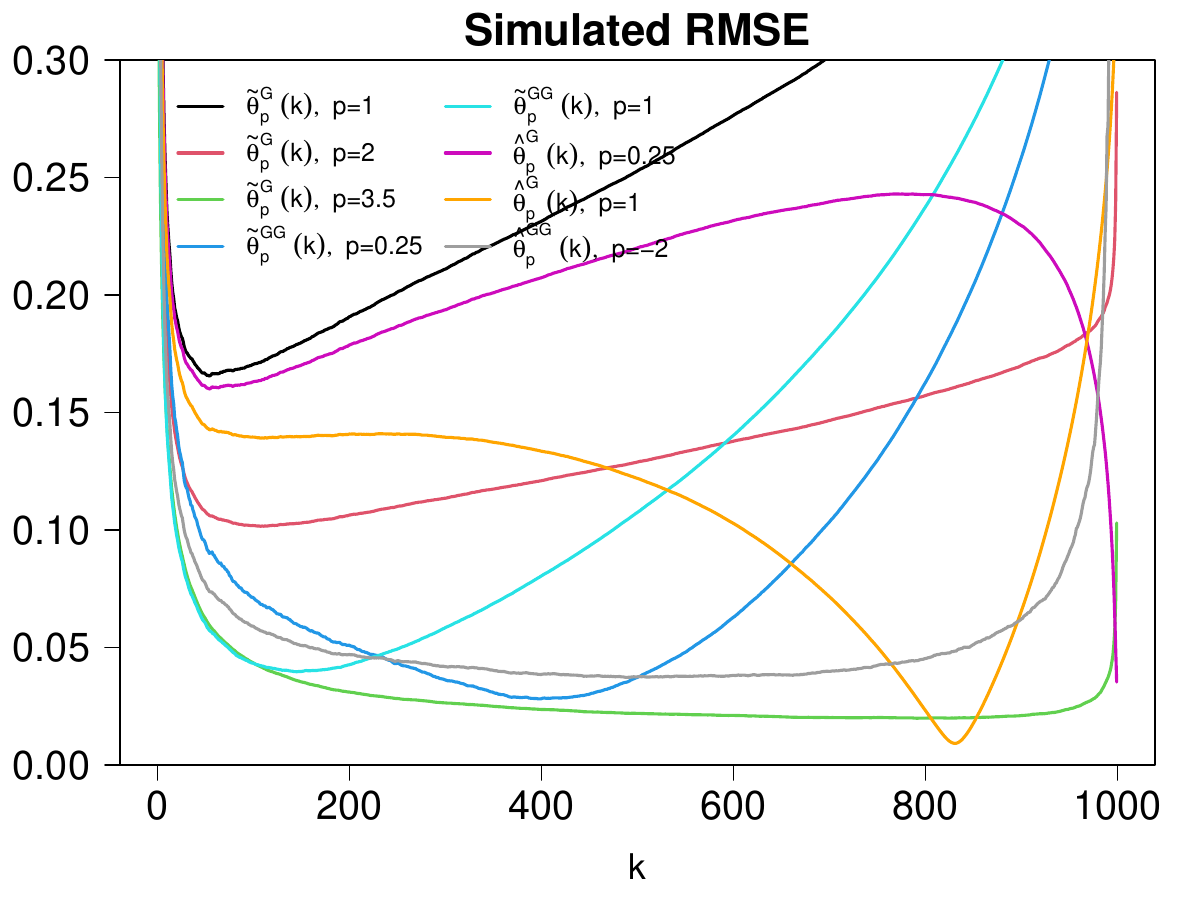}
\caption{Simulated mean values (left) and RMSE (right) for the Half-Normal model.}
\label{fig:fr025}%
\end{figure}

\begin{figure}[h!]
\includegraphics[width=0.45\textwidth]{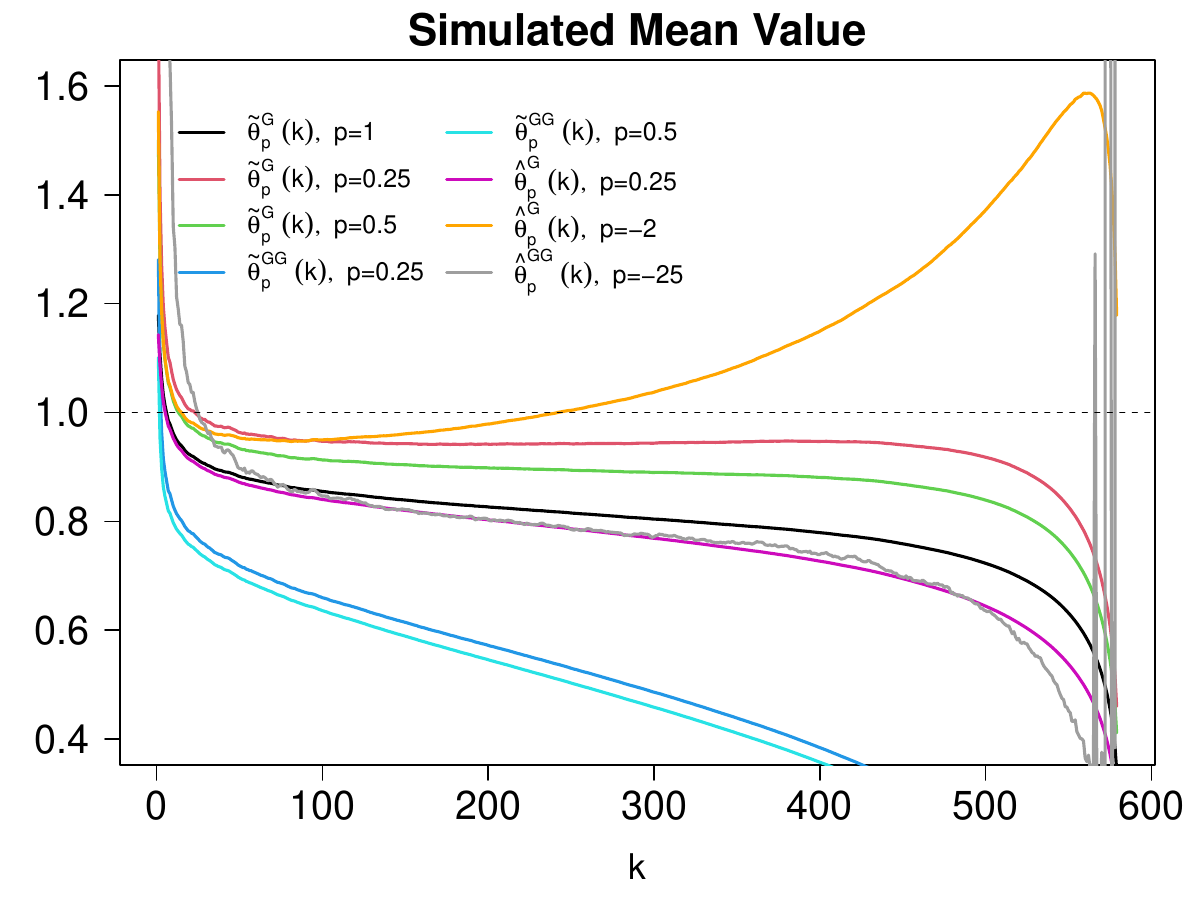}\includegraphics[width=0.45\textwidth]{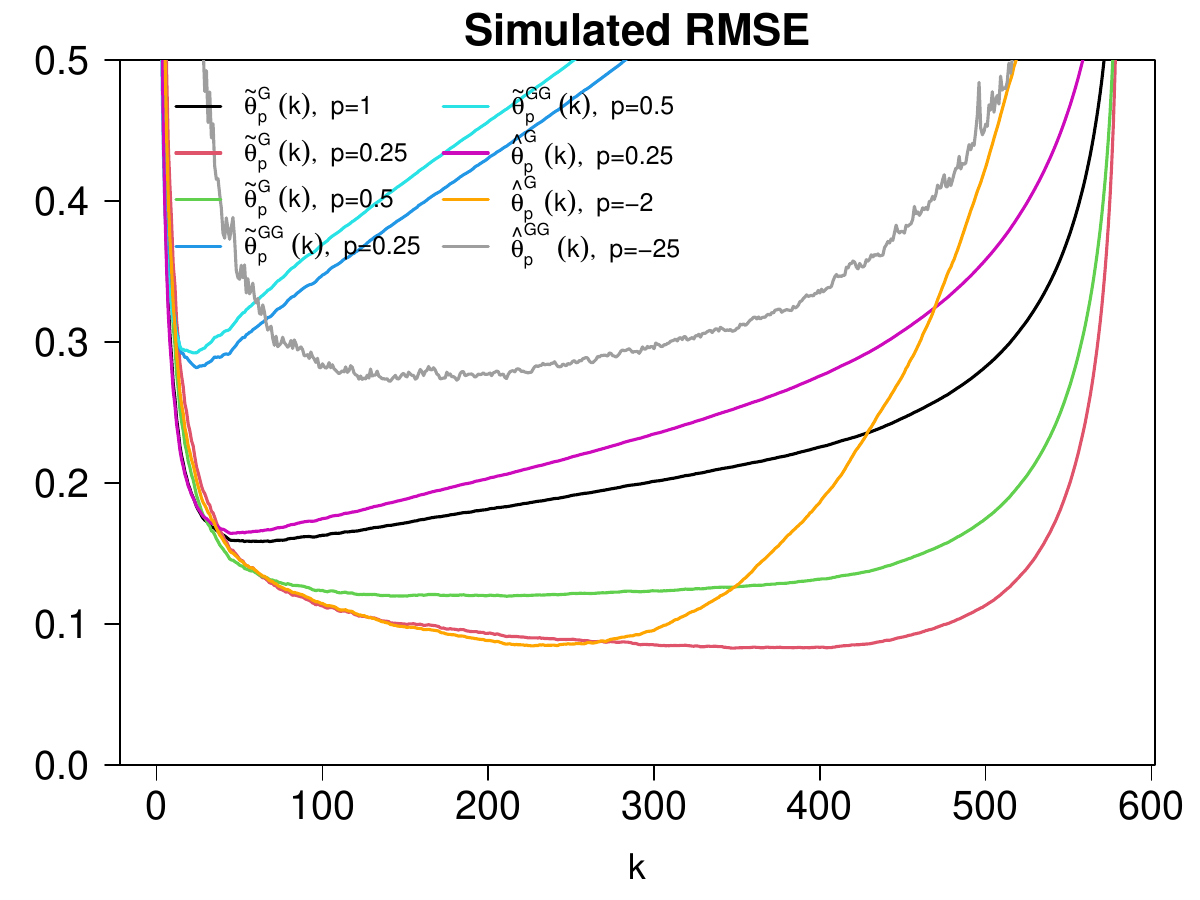}
\caption{Simulated mean values (left) and RMSE (right) for the standard Gumbel model.
}
\label{fig:gumbel}%
\end{figure}


\begin{figure}[h!]
\includegraphics[width=0.45\textwidth]{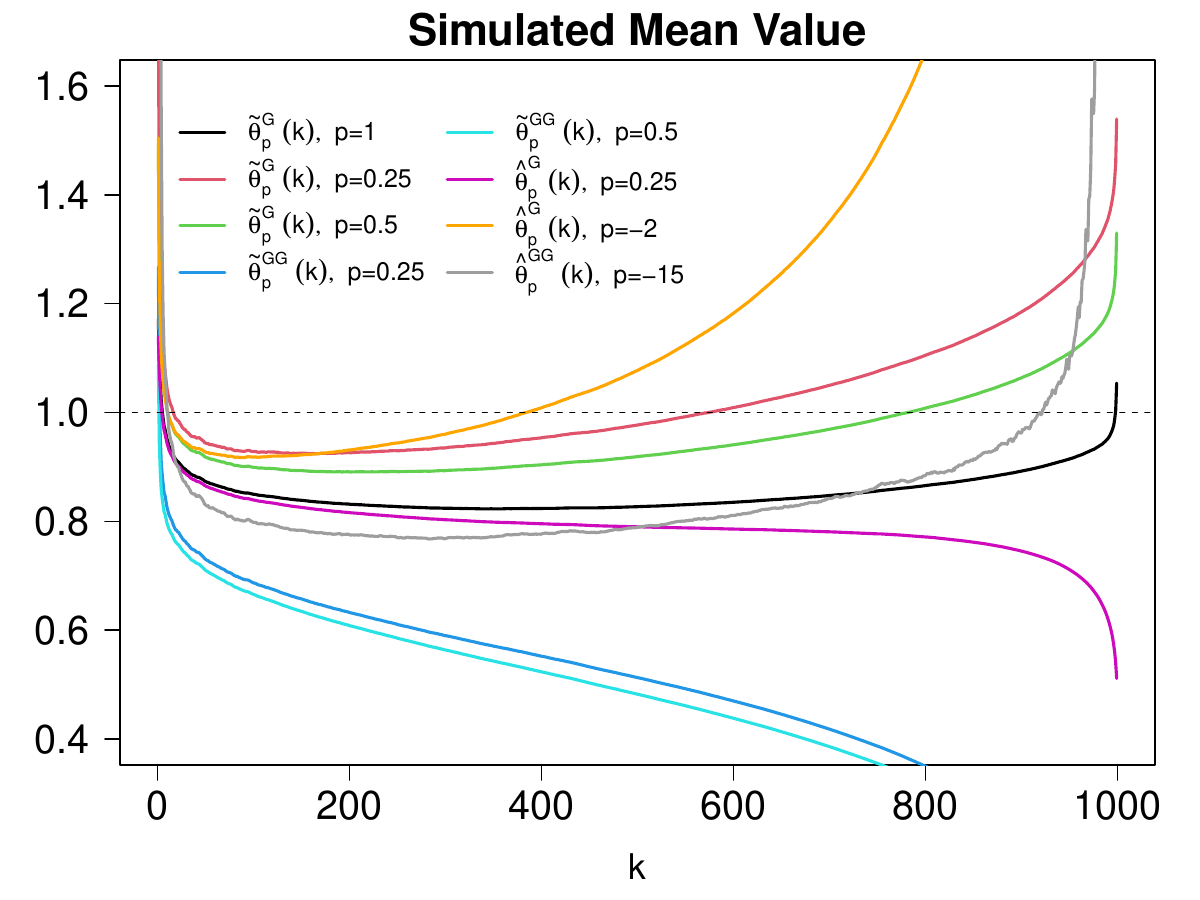}\includegraphics[width=0.45\textwidth]{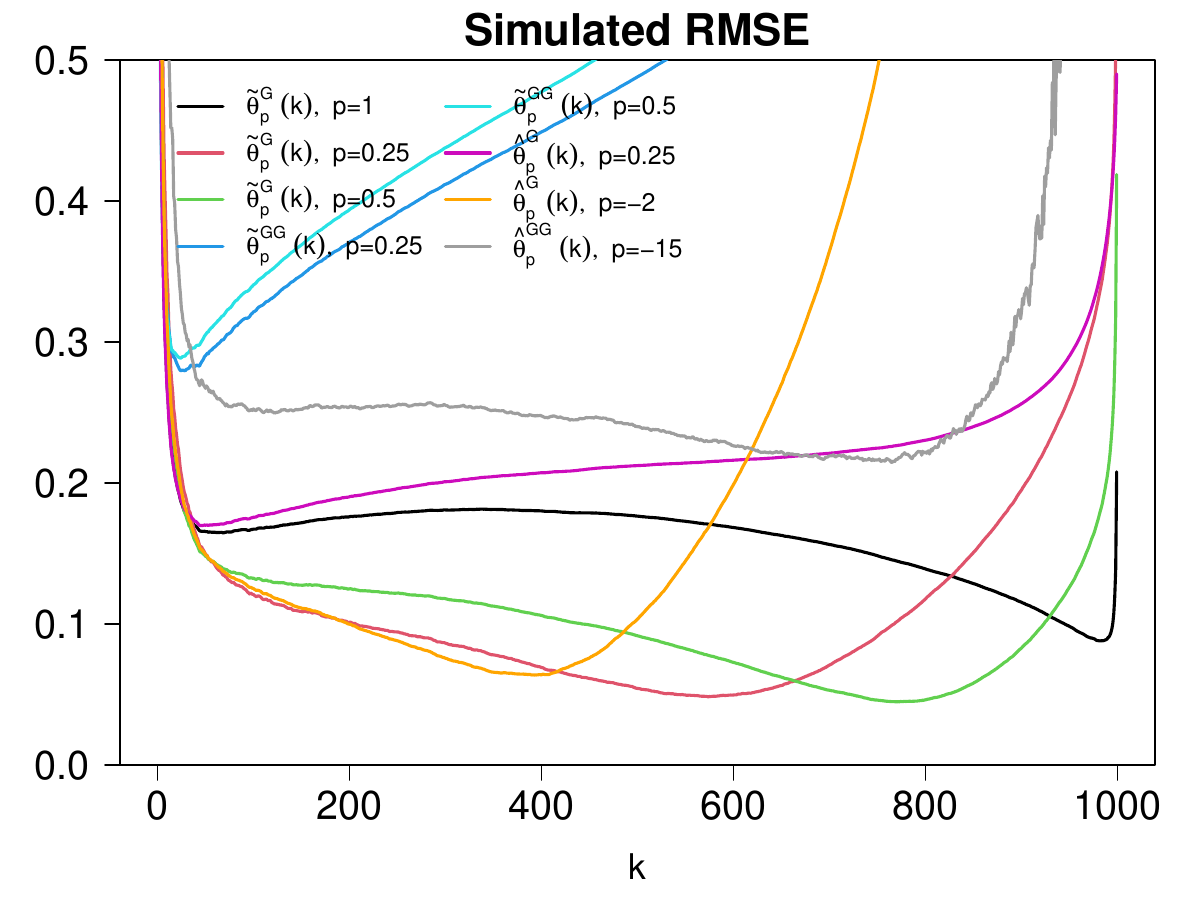}
\caption{Simulated mean values (left) and RMSE (right) for the Half-Logistic model.}
\label{fig:hl}%
\end{figure}

\subsection{Mean value and RMSE at optimal simulated levels}

In Tables \ref{tab_mv} and \ref{tab_rmse}, we present  the simulated values of the mean value (E) and  RMSE, both as given in \eqref{simulE_RMSE_k0}, of the class of estimators under study. For each model, the mean value closest to the target value $\theta$ and the smallest RMSE are highlighted in \textbf{bold}. In general, the results align with the conclusions already discussed in subsection \ref{sec_mean_value_rmse_patterns}. Discrepancies occur either due to similar performance of different estimators or due to the fact that the smallest RMSE is not always associated with the smallest absolute bias.

\begin{longtable}{lrrrrrrr}
\hline \multicolumn{8}{r}{(Continues)} \\ 
\endfoot

\hline 
\endlastfoot

\caption{Simulated mean value of $\widetilde\theta^{\rm G}_p$, $\widetilde\theta^{\rm GG}_p$, $\widehat\theta^{\rm G}_p$ and $\widehat\theta^{\rm GG}_p$},\ computed at the simulated optimal level \label{tab_mv} \\%
  \hline
 & 100 & 200 & 400 & 750 & 1000 & 1500 & 2000 \\ 
  \hline
  & \multicolumn{7}{c}{standard Exponential model\quad ($\theta=1$)}\\ 
  $\widetilde\theta^{\rm G}_p$, $p=1$ & \textbf{1.0137} & \textbf{1.0060} & \textbf{1.0046} & \textbf{1.0022} & \textbf{1.0017} & \textbf{1.0011} & \textbf{1.0007} \\ 
  $\widetilde\theta^{\rm G}_p$, $p=0.75$ & 1.0739 & 1.0623 & 1.0582 & 1.0502 & 1.0495 & 1.0437 & 1.0434 \\ 
  $\widetilde\theta^{\rm G}_p$, $p=1.25$ & 0.9582 & 0.9560 & 0.9584 & 0.9590 & 0.9608 & 0.9663 & 0.9649 \\ 
  $\widetilde\theta^{\rm GG}_p$, $p=0.25$ & 0.7973 & 0.8230 & 0.8356 & 0.8424 & 0.8586 & 0.8599 & 0.8664 \\ 
  $\widetilde\theta^{\rm GG}_p$, $p=1$ & 0.7564 & 0.7703 & 0.7797 & 0.8009 & 0.8054 & 0.8194 & 0.8228 \\ 
  $\widehat\theta^{\rm G}_p$, $p=0.25$ & 0.9408 & 0.9550 & 0.9607 & 0.9692 & 0.9669 & 0.9710 & 0.9725 \\ 
  $\widehat\theta^{\rm G}_p$, $p=0.5$ & 0.9220 & 0.9396 & 0.9431 & 0.9530 & 0.9534 & 0.9622 & 0.9629 \\ 
  $\widehat\theta^{\rm GG}_p$, $p=-10$ & 1.0798 & 0.9817 & 0.9440 & 0.9285 & 0.9242 & 0.9259 & 0.9196 \\  
  \hline%
  & \multicolumn{7}{c}{$Weibull$(2,1) model\quad ($\theta=0.5$)}  \\ 
  $\widetilde\theta^{\rm G}_p$, $p=1$ & \textbf{0.5068} & \textbf{0.5030} & \textbf{0.5023} & \textbf{0.5011} & \textbf{0.5009} & \textbf{0.5005} & \textbf{0.5003} \\ 
  $\widetilde\theta^{\rm G}_p$, $p=0.75$ & 0.5370 & 0.5312 & 0.5291 & 0.5251 & 0.5248 & 0.5219 & 0.5217 \\ 
  $\widetilde\theta^{\rm G}_p$, $p=1.25$ & 0.4791 & 0.4780 & 0.4792 & 0.4795 & 0.4804 & 0.4832 & 0.4824 \\ 
  $\widetilde\theta^{\rm GG}_p$, $p=0.25$ & 0.3986 & 0.4115 & 0.4178 & 0.4212 & 0.4293 & 0.4300 & 0.4332 \\ 
  $\widetilde\theta^{\rm GG}_p$, $p=0.75$ & 0.3853 & 0.3962 & 0.3985 & 0.4086 & 0.4104 & 0.4188 & 0.4208 \\ 
  $\widehat\theta^{\rm G}_p$, $p=-0.25$ & 0.5222 & 0.5197 & 0.5171 & 0.5136 & 0.5127 & 0.5104 & 0.5100 \\ 
  $\widehat\theta^{\rm G}_p$, $p=0.25$ & 0.4838 & 0.4853 & 0.4875 & 0.4886 & 0.4887 & 0.4899 & 0.4902 \\ 
  $\widehat\theta^{\rm GG}_p$, $p=-20$ & 0.5399 & 0.4908 & 0.4720 & 0.4642 & 0.4621 & 0.4630 & 0.4598 \\  
  \hline%
  & \multicolumn{7}{c}{$Gamma$(0.75,1) model\quad ($\theta=1$)}  \\ 
  $\widetilde\theta^{\rm G}_p$, $p=1$     & 1.1875 & 1.1684 & 1.1538 & 1.1376 & 1.1336 & 1.1285 & 1.1257 \\ 
  $\widetilde\theta^{\rm G}_p$, $p=1.75$  & 1.0375 & 1.0320 & 1.0292 & 1.0251 & 1.0241 & 1.0215 & 1.0215 \\ 
  $\widetilde\theta^{\rm G}_p$, $p=2$     & \textbf{0.9917} & \textbf{0.9855} & \textbf{0.9882} & 0.9838 & \textbf{0.9872} & 0.9878 & \textbf{0.9899} \\ 
  $\widetilde\theta^{\rm GG}_p$, $p=0.25$ & 0.8599 & 0.8908 & 0.8963 & 0.9025 & 0.9192 & 0.9189 & 0.9257 \\ 
  $\widetilde\theta^{\rm GG}_p$, $p=0.5$  & 0.8437 & 0.8592 & 0.8775 & 0.8935 & 0.9001 & 0.9075 & 0.9146 \\ 
  $\widehat\theta^{\rm G}_p$, $p=0.25$    & 0.9449 & 0.9648 & 0.9741 & \textbf{0.9870} & 0.9869 & \textbf{0.9895} & 0.9888 \\ 
  $\widehat\theta^{\rm G}_p$, $p=-1.75$   & 0.9269 & 0.9501 & 0.9653 & 0.9749 & 0.9794 & 0.9829 & 0.9861 \\ 
  $\widehat\theta^{\rm GG}_p$, $p=-25$    & 1.0078 & 0.9721 & 0.9580 & 0.9509 & 0.9505 & 0.9484 & 0.9474 \\ 
  \hline%
  & \multicolumn{7}{c}{Half-Normal model\quad ($\theta=0.5$)}  \\ 
  $\widetilde\theta^{\rm G}_p$, $p=1$     & 0.7050 & 0.6868 & 0.6702 & 0.6567 & 0.6478 & 0.6437 & 0.6403 \\
  $\widetilde\theta^{\rm G}_p$, $p=2$     & 0.6196 & 0.6114 & 0.6009 & 0.5932 & 0.5916 & 0.5876 & 0.5862 \\ 
  $\widetilde\theta^{\rm G}_p$, $p=3.5$   & 0.5146 & 0.5123 & 0.5116 & 0.5100 & 0.5088 & 0.5065 & 0.5072 \\ 
  $\widetilde\theta^{\rm GG}_p$, $p=0.25$ & 0.4738 & 0.4844 & 0.4937 & 0.4926 & 0.4967 & 0.4978 & 0.4978 \\ 
  $\widetilde\theta^{\rm GG}_p$, $p=1$    & 0.4548 & 0.4590 & 0.4716 & 0.4735 & 0.4797 & 0.4804 & 0.4854 \\ 
  $\widehat\theta^{\rm G}_p$, $p=0.25$    & 0.6444 & 0.5981 & 0.5577 & 0.5239 & 0.5092 & 0.4896 & 0.5031 \\ 
  $\widehat\theta^{\rm G}_p$, $p=1$       & \textbf{0.4922} & \textbf{0.4966} & \textbf{0.4983} & \textbf{0.4996} &\textbf{ 0.4993} & \textbf{0.4995} & \textbf{0.4996} \\ 
  $\widehat\theta^{\rm GG}_p$, $p=-2$     & 0.5361 & 0.5251 & 0.5189 & 0.5176 & 0.5166 & 0.5142 & 0.5145 \\ 
  \hline%
  & \multicolumn{7}{c}{standard Gumbel model\quad $\mu=0$ ($\theta=1$)}  \\ 
  $\widetilde\theta^{\rm G}_p$, $p=1$     & 0.7763 & 0.8148 & 0.8421 & 0.8684 & 0.8802 & 0.8839 & 0.8875 \\ 
  $\widetilde\theta^{\rm G}_p$, $p=0.25$  & 0.8379 & 0.8787 & 0.8959 & 0.9256 & 0.9460 & 0.9582 & 0.9801 \\ 
  $\widetilde\theta^{\rm G}_p$, $p=0.5$   & 0.8090 & 0.8626 & 0.8843 & 0.8974 & 0.8976 & 0.9082 & 0.9105 \\ 
  $\widetilde\theta^{\rm GG}_p$, $p=0.25$ & 0.6908 & 0.7197 & 0.7351 & 0.7623 & 0.7705 & 0.7812 & 0.7902 \\ 
  $\widetilde\theta^{\rm GG}_p$, $p=0.5$  & 0.6522 & 0.7112 & 0.7168 & 0.7452 & 0.7523 & 0.7713 & 0.7730 \\ 
  $\widehat\theta^{\rm G}_p$, $p=0.25$    & 0.7666 & 0.8247 & 0.8424 & 0.8667 & 0.8777 & 0.8764 & 0.8778 \\ 
  $\widehat\theta^{\rm G}_p$, $p=-2$      & \textbf{0.8762} & \textbf{0.9479} & \textbf{0.9778} & \textbf{0.9911} & \textbf{0.9915} & \textbf{1.0013} & \textbf{0.9973} \\ 
  $\widehat\theta^{\rm GG}_p$, $p=-25$    & 1.9545 & 0.9363 & 0.8233 & 0.8024 & 0.8222 & 0.8218 & 0.8248 \\ 
  \hline%
  & \multicolumn{7}{c}{Half-Logistic model\quad ($\theta=1$)}  \\  
  $\widetilde\theta^{\rm G}_p$, $p=1$     & 0.9024 & 0.9252 & 0.9359 & 0.9392 & 0.9397 & 0.9499 & 0.9556 \\ 
  $\widetilde\theta^{\rm G}_p$, $p=0.25$  & 1.0190 & \textbf{1.0048} & \textbf{1.0019} & 0.9974 & \textbf{1.0000} & 1.0005 & \textbf{1.0000} \\ 
  $\widetilde\theta^{\rm G}_p$, $p=0.5$   & 0.9742 & 0.9861 & 0.9952 & 0.9986 & 0.9959 & 0.9993 & 0.9987 \\ 
  $\widetilde\theta^{\rm GG}_p$, $p=0.25$ & 0.7039 & 0.7201 & 0.7421 & 0.7514 & 0.7635 & 0.7793 & 0.7805 \\ 
  $\widetilde\theta^{\rm GG}_p$, $p=0.5$  & 0.6897 & 0.7166 & 0.7354 & 0.7559 & 0.7519 & 0.7716 & 0.7658 \\ 
  $\widehat\theta^{\rm G}_p$, $p=0.25$    & 0.8109 & 0.8372 & 0.8402 & 0.8589 & 0.8709 & 0.8664 & 0.8779 \\ 
  $\widehat\theta^{\rm G}_p$, $p=-2$      & \textbf{1.0173} & 1.0143 & 1.0020 & \textbf{1.0006} & 1.0048 & \textbf{1.0003} & 1.0007 \\ 
  $\widehat\theta^{\rm GG}_p$, $p=-15$    & 0.9496 & 0.8442 & 0.8549 & 0.8365 & 0.8712 & 0.8582 & 0.8690 \\  
   \hline%
\end{longtable}   

\begin{longtable}{lrrrrrrr}
\hline \multicolumn{8}{r}{(Continues)} \\ 
\endfoot

\hline 
\endlastfoot

\caption{Simulated RMSE of $\widetilde\theta^{\rm G}_p$, $\widetilde\theta^{\rm GG}_p$, $\widehat\theta^{\rm G}_p$ and $\widehat\theta^{\rm GG}_p$}, computed at the simulated optimal level \label{tab_rmse} \\%
  \hline
 & 100 & 200 & 400 & 750 & 1000 & 1500 & 2000 \\ 
  \hline
  & \multicolumn{7}{c}{standard Exponential model\quad ($\theta=1$)}\\ 
  $\widetilde\theta^{\rm G}_p$, $p=1$  & \textbf{0.1223} & \textbf{0.0857} & \textbf{0.0604} & \textbf{0.0440} & \textbf{0.0379} & \textbf{0.0305} & \textbf{0.0269} \\ 
  $\widetilde\theta^{\rm G}_p$, $p=0.75$ & 0.1554 & 0.1152 & 0.0909 & 0.0739 & 0.0683 & 0.0612 & 0.0570 \\ 
  $\widetilde\theta^{\rm G}_p$, $p=1.25$ & 0.1185 & 0.0911 & 0.0716 & 0.0601 & 0.0560 & 0.0510 & 0.0482 \\ 
  $\widetilde\theta^{\rm GG}_p$, $p=0.25$ & 0.2799 & 0.2470 & 0.2170 & 0.1962 & 0.1890 & 0.1766 & 0.1706 \\ 
  $\widetilde\theta^{\rm GG}_p$, $p=1$ & 0.3130 & 0.2808 & 0.2574 & 0.2351 & 0.2269 & 0.2136 & 0.2061 \\ 
  $\widehat\theta^{\rm G}_p$, $p=0.25$ & 0.1189 & 0.0921 & 0.0719 & 0.0593 & 0.0544 & 0.0474 & 0.0443 \\ 
  $\widehat\theta^{\rm G}_p$, $p=0.5$ & 0.1330 & 0.1089 & 0.0890 & 0.0740 & 0.0699 & 0.0612 & 0.0573 \\ 
  $\widehat\theta^{\rm GG}_p$, $p=-10$ & 0.5601 & 0.2976 & 0.1997 & 0.1525 & 0.1368 & 0.1208 & 0.1131 \\ 
  \hline  
  & \multicolumn{7}{c}{$Weibull$(2,1) model \quad ($\theta=0.5$)}  \\ 
  $\widetilde\theta^{\rm G}_p$, $p=1$   & \textbf{0.0612} & \textbf{0.0429} & \textbf{0.0302} & \textbf{0.0220} & \textbf{0.0189} & \textbf{0.0153} & \textbf{0.0135} \\ 
  $\widetilde\theta^{\rm G}_p$, $p=0.75$   & 0.0777 & 0.0576 & 0.0454 & 0.0369 & 0.0342 & 0.0306 & 0.0285 \\ 
  $\widetilde\theta^{\rm G}_p$, $p=1.25$   & 0.0593 & 0.0455 & 0.0358 & 0.0301 & 0.0280 & 0.0255 & 0.0241 \\ 
  $\widetilde\theta^{\rm GG}_p$, $p=0.25$  & 0.1399 & 0.1235 & 0.1085 & 0.0981 & 0.0945 & 0.0883 & 0.0853 \\ 
  $\widetilde\theta^{\rm GG}_p$, $p=0.75$  & 0.1503 & 0.1338 & 0.1214 & 0.1109 & 0.1070 & 0.1008 & 0.0970 \\ 
  $\widehat\theta^{\rm G}_p$, $p=-0.25$     & 0.0723 & 0.0516 & 0.0381 & 0.0289 & 0.0258 & 0.0221 & 0.0198 \\ 
  $\widehat\theta^{\rm G}_p$, $p=0.25$     & 0.0565 & 0.0423 & 0.0318 & 0.0249 & 0.0224 & 0.0196 & 0.0178 \\ 
  $\widehat\theta^{\rm GG}_p$, $p=-20$    & 0.2801 & 0.1488 & 0.0999 & 0.0762 & 0.0684 & 0.0604 & 0.0565 \\  
  \hline  
  & \multicolumn{7}{c}{$Gamma$(0.75,1) model\quad ($\theta=1$)}  \\ 
  $\widetilde\theta^{\rm G}_p$, $p=1$  & 0.2470 & 0.2107 & 0.1839 & 0.1626 & 0.1547 & 0.1456 & 0.1408 \\ 
  $\widetilde\theta^{\rm G}_p$, $p=1.75$ & 0.1247 & 0.0895 & 0.0653 & 0.0495 & 0.0448 & 0.0384 & 0.0350 \\ 
  $\widetilde\theta^{\rm G}_p$, $p=2$  & 0.1115 & \textbf{0.0794} & \textbf{0.0568} & \textbf{0.0430} & \textbf{0.0378} & \textbf{0.0318} & \textbf{0.0279} \\ 
  $\widetilde\theta^{\rm GG}_p$, $p=0.25$ & 0.2180 & 0.1814 & 0.1549 & 0.1326 & 0.1246 & 0.1127 & 0.1044 \\ 
  $\widetilde\theta^{\rm GG}_p$, $p=0.5$ & 0.2284 & 0.1948 & 0.1688 & 0.1464 & 0.1374 & 0.1263 & 0.1178 \\ 
  $\widehat\theta^{\rm G}_p$, $p=0.25$    & \textbf{0.1087} & 0.0834 & 0.0644 & 0.0486 & 0.0429 & 0.0357 & 0.0316 \\ 
  $\widehat\theta^{\rm G}_p$, $p=-1.75$    & 0.1222 & 0.0956 & 0.0750 & 0.0584 & 0.0522 & 0.0440 & 0.0392 \\ 
  $\widehat\theta^{\rm GG}_p$, $p=-25$   & 0.2760 & 0.1869 & 0.1326 & 0.1034 & 0.0932 & 0.0832 & 0.0771 \\ 
  \hline
  & \multicolumn{7}{c}{Half-Normal model\quad ($\theta=0.5$)}  \\ 
  $\widetilde\theta^{\rm G}_p$, $p=1$  & 0.2456 & 0.2152 & 0.1906 & 0.1735 & 0.1655 & 0.1576 & 0.1517 \\  
  $\widetilde\theta^{\rm G}_p$, $p=2$  & 0.1503 & 0.1303 & 0.1157 & 0.1059 & 0.1016 & 0.0973 & 0.0941 \\ 
  $\widetilde\theta^{\rm G}_p$, $p=3.5$  & 0.0572 & 0.0410 & 0.0301 & 0.0226 & 0.0199 & 0.0169 & 0.0150 \\ 
  $\widetilde\theta^{\rm GG}_p$, $p=0.25$ & 0.0800 & 0.0592 & 0.0430 & 0.0321 & 0.0281 & 0.0230 & 0.0200 \\ 
  $\widetilde\theta^{\rm GG}_p$, $p=1$ & 0.0877 & 0.0698 & 0.0544 & 0.0443 & 0.0396 & 0.0346 & 0.0311 \\ 
  $\widehat\theta^{\rm G}_p$, $p=0.25$    & 0.1703 & 0.1186 & 0.0759 & 0.0450 & 0.0353 & 0.0308 & 0.0249 \\ 
  $\widehat\theta^{\rm G}_p$, $p=1$    & \textbf{0.0289} & \textbf{0.0207} & \textbf{0.0147} & \textbf{0.0108} & \textbf{0.0092} & \textbf{0.0076} & \textbf{0.0066} \\ 
  $\widehat\theta^{\rm GG}_p$, $p=-2$   & 0.1178 & 0.0814 & 0.0572 & 0.0420 & 0.0372 & 0.0313 & 0.0280 \\ 
   \hline
  & \multicolumn{7}{c}{Gumbel model \quad ($\theta=1$)}  \\ 
  $\widetilde\theta^{\rm G}_p$, $p=1$     & 0.2980 & 0.2357 & 0.1976 & 0.1702 & 0.1585 & 0.1483 & 0.1384 \\
  $\widetilde\theta^{\rm G}_p$, $p=0.25$  & \textbf{0.2627} & \textbf{0.1901} & 0.1441 & 0.1038 & \textbf{0.0829} & \textbf{0.0673} & \textbf{0.0515} \\ 
  $\widetilde\theta^{\rm G}_p$, $p=0.5$   & 0.2735 & 0.2033 & 0.1649 & 0.1335 & 0.1197 & 0.1116 & 0.0983 \\ 
  $\widetilde\theta^{\rm GG}_p$, $p=0.25$ & 0.4229 & 0.3667 & 0.3265 & 0.2970 & 0.2818 & 0.2673 & 0.2549 \\ 
  $\widetilde\theta^{\rm GG}_p$, $p=0.5$  & 0.4290 & 0.3734 & 0.3371 & 0.3052 & 0.2923 & 0.2755 & 0.2648 \\ 
  $\widehat\theta^{\rm G}_p$, $p=0.25$    & 0.3044 & 0.2436 & 0.2056 & 0.1770 & 0.1643 & 0.1544 & 0.1451 \\ 
  $\widehat\theta^{\rm G}_p$, $p=-2$      & 0.2807 & 0.1951 & \textbf{0.1339} & \textbf{0.0968} & 0.0845 & 0.0695 & 0.0593 \\ 
  $\widehat\theta^{\rm GG}_p$, $p=-25$   & 10.0605 & 0.9215 & 0.4022 & 0.3007 & 0.2722 & 0.2444 & 0.2299 \\ 
   \hline
  & \multicolumn{7}{c}{Half-Logistic model\quad ($\theta=1$)}  \\ 
  $\widetilde\theta^{\rm G}_p$, $p=1$  & 0.1608 & 0.1310 & 0.1092 & 0.0922 & 0.0880 & 0.0795 & 0.0751 \\ 
  $\widetilde\theta^{\rm G}_p$, $p=0.25$  & 0.1543 & 0.1095 & 0.0764 & 0.0564 & 0.0484 & 0.0391 & 0.0345 \\ 
  $\widetilde\theta^{\rm G}_p$, $p=0.5$  & \textbf{0.1395} & \textbf{0.0987} & \textbf{0.0699} & \textbf{0.0515} & \textbf{0.0449} & \textbf{0.0364} & \textbf{0.0318} \\ 
  $\widetilde\theta^{\rm GG}_p$, $p=0.25$ & 0.3641 & 0.3347 & 0.3091 & 0.2884 & 0.2795 & 0.2678 & 0.2582 \\ 
  $\widetilde\theta^{\rm GG}_p$, $p=0.5$ & 0.3747 & 0.3449 & 0.3180 & 0.2981 & 0.2885 & 0.2758 & 0.2675 \\ 
  $\widehat\theta^{\rm G}_p$, $p=0.25$    & 0.2210 & 0.2047 & 0.1896 & 0.1759 & 0.1698 & 0.1626 & 0.1571 \\ 
  $\widehat\theta^{\rm G}_p$, $p=-2$    & 0.2179 & 0.1488 & 0.1034 & 0.0736 & 0.0638 & 0.0519 & 0.0453 \\ 
  $\widehat\theta^{\rm GG}_p$, $p=-15$   & 0.5890 & 0.3313 & 0.2658 & 0.2297 & 0.2144 & 0.1948 & 0.1843 \\ 
   \hline
\end{longtable}   


\section{Conclusions}
In this paper we have introduced new classes of  estimators for the WTC parameter. These classes are constructed based on the generalized means  of the log-excesses or the relative excesses and incorporate a control parameter denoted as $p$. 
The results from the large-scale simulation experiment show that the  parameter $p$ enables us to have, in general, a near zero bias and a high stability of the estimates, as a function of the number $k$ of upper OS’s considered -— a crucial aspect for practical applicability. In addition, and with an adequate choice of the control parameter, the classes of  estimators exhibit  a reasonably flat RMSE pattern, as a function of $k$. Moreover, the results confirm improvements over existing methodologies. {Applications to real data are crucial, but are out of the scope of this article.}

\section*{Acknowledgments}
{
This work is funded by national funds through the FCT – Fundação para a Ciência e a Tecnologia, I.P., under the scope of the projects UIDB/00297/2020 (https://doi.org/10.54499/UIDB/00297/2020) and UIDP/00297/2020 (https://doi.org/10.54499/UIDP/00297/2020) (Center for Mathematics and Applications), 
project  UIDB/00006/2020 (https://doi.org/10.54499/UIDB/00006/2020) (CEAUL), project  UIDB/MAT/04674/2020 (https://doi.org/10.54499/UIDB/04674/2020) (CIMA). The first author is also financed by project 10.54499/CEECINST/00054/2018/CP1522/CT0003  (https://doi.org/10.54499/CEECINST/00054/2018/CP1522/CT0003).}
{The authors also acknowledge the constructive suggestions from the anonymous reviewers.}

\appendix

\section{Simulated optimal sample fraction\label{app1}}

This appendix contains in Table \ref{tab_osf} the simulated optimal sample fraction, i.e., the simulated optimum level in \eqref{eq:osf} divided by the sample size. The optimal sample fraction is presented in appendix to improve readability of the main text.

\begin{longtable}{rrrrrrrr}
\hline \multicolumn{8}{r}{(Continues)} \\ 
\endfoot

\hline 
\endlastfoot

\caption{Simulated optimal sample fraction (OSF) of $\widetilde\theta^{\rm G}_p$, $\widetilde\theta^{\rm GG}_p$, $\widehat\theta^{\rm G}_p$ and $\widehat\theta^{\rm GG}_p$}
\label{tab_osf} \\%
  \hline
 & 100 & 200 & 400 & 750 & 1000 & 1500 & 2000 \\ 
  \hline
  & \multicolumn{7}{c}{standard Exponential model\quad ($\theta=1$)}\\ 
  $\widetilde\theta^{\rm G}_p$, $p=1$  & 0.6800 & 0.6950 & 0.6700 & 0.7000 & 0.6980 & 0.6993 & 0.7055 \\ 
  $\widetilde\theta^{\rm G}_p$, $p=0.75$  & 0.6400 & 0.5800 & 0.5350 & 0.4320 & 0.4230 & 0.3273 & 0.3390 \\ 
  $\widetilde\theta^{\rm G}_p$, $p=1.25$  & 0.6800 & 0.5950 & 0.5275 & 0.4680 & 0.4200 & 0.2940 & 0.3040 \\ 
  $\widetilde\theta^{\rm GG}_p$, $p=0.25$ & 0.1900 & 0.1250 & 0.0925 & 0.0747 & 0.0500 & 0.0460 & 0.0370 \\ 
  $\widetilde\theta^{\rm GG}_p$, $p=1$ & 0.1100 & 0.0800 & 0.0625 & 0.0400 & 0.0350 & 0.0260 & 0.0225 \\ 
  $\widehat\theta^{\rm G}_p$, $p=0.25$    & 0.5900 & 0.4800 & 0.4225 & 0.3240 & 0.3380 & 0.2940 & 0.2675 \\ 
  $\widehat\theta^{\rm G}_p$, $p=0.5$    & 0.4100 & 0.3200 & 0.2875 & 0.2200 & 0.2110 & 0.1653 & 0.1580 \\ 
  $\widehat\theta^{\rm GG}_p$, $p=-10$   & 0.4400 & 0.3800 & 0.4325 & 0.3440 & 0.4130 & 0.2940 & 0.3860 \\   
  \hline%
  & \multicolumn{7}{c}{$Weibull$(2,1) model \quad ($\theta=0.5$)}  \\ 
  $\widetilde\theta^{\rm G}_p$, $p=1$  & 0.6800 & 0.6950 & 0.6700 & 0.7000 & 0.6980 & 0.6993 & 0.7055 \\ 
  $\widetilde\theta^{\rm G}_p$, $p=0.75$  & 0.6400 & 0.5800 & 0.5350 & 0.4320 & 0.4230 & 0.3273 & 0.3390 \\ 
  $\widetilde\theta^{\rm G}_p$, $p=1.25$  & 0.6800 & 0.5950 & 0.5275 & 0.4680 & 0.4200 & 0.2940 & 0.3040 \\ 
  $\widetilde\theta^{\rm GG}_p$, $p=0.25$ & 0.1900 & 0.1250 & 0.0925 & 0.0747 & 0.0500 & 0.0460 & 0.0370 \\ 
  $\widetilde\theta^{\rm GG}_p$, $p=0.75$ & 0.1300 & 0.0850 & 0.0700 & 0.0453 & 0.0390 & 0.0273 & 0.0225 \\ 
  $\widehat\theta^{\rm G}_p$, $p=-0.25$    & 0.5500 & 0.5800 & 0.5350 & 0.4653 & 0.4480 & 0.3800 & 0.3740 \\ 
  $\widehat\theta^{\rm G}_p$, $p=0.25$    & 0.6800 & 0.5950 & 0.5325 & 0.4733 & 0.4690 & 0.4200 & 0.4045 \\ 
  $\widehat\theta^{\rm GG}_p$, $p=-20$   & 0.4400 & 0.3800 & 0.4325 & 0.3440 & 0.4130 & 0.2940 & 0.3860 \\ 
  \hline%
  & \multicolumn{7}{c}{$Gamma$(0.75,1) model \quad ($\theta=1$)}  \\ 
  $\widetilde\theta^{\rm G}_p$, $p=1$ & 0.4400 & 0.3350 & 0.2550 & 0.1680 & 0.1580 & 0.1320 & 0.1170 \\ 
  $\widetilde\theta^{\rm G}_p$, $p=1.75$ & 0.6900 & 0.7000 & 0.7200 & 0.7720 & 0.7810 & 0.8287 & 0.8345 \\ 
  $\widetilde\theta^{\rm G}_p$, $p=2$ & 0.6800 & 0.7000 & 0.6075 & 0.6560 & 0.5880 & 0.5560 & 0.5185 \\ 
  $\widetilde\theta^{\rm GG}_p$, $p=0.25$ & 0.2900 & 0.2050 & 0.1700 & 0.1440 & 0.1030 & 0.1007 & 0.0855 \\ 
  $\widetilde\theta^{\rm GG}_p$, $p=0.5$ & 0.2400 & 0.1800 & 0.1300 & 0.0960 & 0.0810 & 0.0667 & 0.0555 \\ 
  $\widehat\theta^{\rm G}_p$, $p=0.25$ & 0.3600 & 0.3150 & 0.2875 & 0.2507 & 0.2490 & 0.2407 & 0.2415 \\ 
  $\widehat\theta^{\rm G}_p$, $p=-1.75$ & 0.2900 & 0.2450 & 0.2100 & 0.1853 & 0.1740 & 0.1647 & 0.1565 \\ 
  $\widehat\theta^{\rm GG}_p$, $p=-25$ & 0.3800 & 0.4600 & 0.4650 & 0.4947 & 0.3500 & 0.3440 & 0.3360 \\ 
  \hline%
  & \multicolumn{7}{c}{Half-Normal model\quad ($\theta=0.5$)}  \\
  $\widetilde\theta^{\rm G}_p$, $p=1$  & 0.2100 & 0.1550 & 0.1125 & 0.0760 & 0.0540 & 0.0480 & 0.0445 \\ 
  $\widetilde\theta^{\rm G}_p$, $p=2$  & 0.3100 & 0.2650 & 0.1700 & 0.1147 & 0.1110 & 0.0793 & 0.0745 \\ 
  $\widetilde\theta^{\rm G}_p$, $p=3.5$  & 0.6800 & 0.7050 & 0.6550 & 0.7293 & 0.7910 & 0.8660 & 0.8430 \\ 
  $\widetilde\theta^{\rm GG}_p$, $p=0.25$ & 0.5300 & 0.4750 & 0.4225 & 0.4187 & 0.3970 & 0.3920 & 0.3920 \\ 
  $\widetilde\theta^{\rm GG}_p$, $p=1$ & 0.3000 & 0.2600 & 0.1925 & 0.1787 & 0.1450 & 0.1400 & 0.1145 \\ 
  $\widehat\theta^{\rm G}_p$, $p=0.25$    & 0.9900 & 0.9950 & 0.9975 & 0.9987 & 0.9990 & 0.9993 & 0.9990 \\ 
  $\widehat\theta^{\rm G}_p$, $p=1$    & 0.8400 & 0.8350 & 0.8325 & 0.8307 & 0.8310 & 0.8307 & 0.8305 \\ 
  $\widehat\theta^{\rm GG}_p$, $p=-2$   & 0.5500 & 0.5000 & 0.5425 & 0.4680 & 0.4880 & 0.5840 & 0.5585 \\   
  \hline%
  & \multicolumn{7}{c}{standard Gumbel model\quad ($\theta=1$)}  \\
  $\widetilde\theta^{\rm G}_p$, $p=1$    & 0.1200 & 0.1200 & 0.0925 & 0.0640 & 0.0530 & 0.0453 & 0.0445 \\ 
  $\widetilde\theta^{\rm G}_p$, $p=0.25$ & 0.1900 & 0.2250 & 0.2675 & 0.3013 & 0.3480 & 0.4027 & 0.4145 \\ 
  $\widetilde\theta^{\rm G}_p$, $p=0.5$  & 0.1700 & 0.1600 & 0.1350 & 0.1387 & 0.2110 & 0.1293 & 0.3905 \\ 
  $\widetilde\theta^{\rm G}_p$, $p=0.25$ & 0.0600 & 0.0500 & 0.0400 & 0.0267 & 0.0240 & 0.0187 & 0.0160 \\ 
  $\widetilde\theta^{\rm G}_p$, $p=0.5$  & 0.0600 & 0.0400 & 0.0350 & 0.0240 & 0.0220 & 0.0153 & 0.0145 \\ 
  $\widehat\theta^{\rm G}_p$, $p=0.25$   & 0.1100 & 0.0850 & 0.0725 & 0.0520 & 0.0450 & 0.0427 & 0.0440 \\ 
  $\widehat\theta^{\rm G}_p$, $p=-2$     & 0.2200 & 0.2950 & 0.2925 & 0.2573 & 0.2270 & 0.2400 & 0.2155 \\ 
  $\widehat\theta^{\rm GG}_p$, $p=-25$   & 0.2200 & 0.2250 & 0.2025 & 0.1907 & 0.1410 & 0.1247 & 0.1175 \\ 
  \hline%
  & \multicolumn{7}{c}{Half-Logistic model\quad ($\theta=1$)}  \\
  $\widetilde\theta^{\rm G}_p$, $p=1$ & 0.8800 & 0.9450 & 0.9700 & 0.9800 & 0.9820 & 0.9900 & 0.9935 \\ 
  $\widetilde\theta^{\rm G}_p$, $p=0.25$ & 0.5800 & 0.5650 & 0.5725 & 0.5613 & 0.5740 & 0.5787 & 0.5790 \\ 
  $\widetilde\theta^{\rm G}_p$, $p=0.5$ & 0.6700 & 0.7300 & 0.7625 & 0.7773 & 0.7710 & 0.7827 & 0.7830 \\ 
  $\widetilde\theta^{\rm GG}_p$, $p=0.25$ & 0.1200 & 0.0850 & 0.0525 & 0.0387 & 0.0290 & 0.0193 & 0.0180 \\ 
  $\widetilde\theta^{\rm GG}_p$, $p=0.5$ & 0.1000 & 0.0600 & 0.0400 & 0.0240 & 0.0240 & 0.0153 & 0.0160 \\ 
  $\widehat\theta^{\rm G}_p$, $p=0.25$ & 0.3500 & 0.1600 & 0.1200 & 0.0640 & 0.0450 & 0.0427 & 0.0285 \\ 
  $\widehat\theta^{\rm G}_p$, $p=-2$ & 0.3400 & 0.3750 & 0.3750 & 0.3813 & 0.3930 & 0.3880 & 0.3900 \\ 
  $\widehat\theta^{\rm GG}_p$, $p=-15$ & 0.3900 & 0.3750 & 0.6300 & 0.6680 & 0.7650 & 0.7627 & 0.8075 \\ 
   \hline
\end{longtable}

\bibliography{references}

\end{document}